\documentclass{au}
 \presentedby{\dots}
 \received{\dots}{\dots}

\usepackage{amsmath,amssymb,url}
\usepackage{epic}
\usepackage{eepic}
\newsavebox{\wwide}
\newcommand{\wwidehat}[1]{\sbox{\wwide}{$#1$}
\ifdim\wd\wwide < 1.1 em \widehat{#1} \else
\setlength
{\unitlength}{0.01\wd\wwide}\overset
{\begin{picture}(100,6)
\path(0,0)(50,6)(100,0)
\end{picture}}{#1}\fi}

\newcommand{\wwidetilde}[1]{\sbox{\wwide}{$#1$}
\ifdim\wd\wwide < 1.1 em \widetilde{#1} \else
\setlength
{\unitlength}{0.01\wd\wwide}\overset
{\begin{picture}(100,6)
\path(0,0)(33,6)(45,6)(55,0)(67,0)(100,6)
\end{picture}}{#1}\fi}

\ifx\pdfoutput\@undefined\usepackage[usenames,dvips]{color}
\else\usepackage[usenames,dvipsnames]{color}
\IfFileExists{pdfcolmk.sty}{\usepackage{pdfcolmk}}{} 
\fi

\definecolor{ChadDarkBlue}{rgb}{.1,0,.2}  
\definecolor{ChadBlue}{rgb}{.1,.1,.5}  
\definecolor{ChadRoyal}{rgb}{.2,.2,.8}  
\definecolor{ChadGreen}{rgb}{0,.4,0}    
\definecolor{ChadRed}{rgb}{.5,0,.5}  

\def\zmena#1{{#1}} 
\def\autori#1{{#1}} 

\usepackage{amsthm}
\usepackage{epic,eepic}
\usepackage[mathscr]{euscript}
\usepackage{enumerate}

\usepackage{graphicx}

\usepackage{lineno}
\usepackage{comment}

\def\smallskip{\vskip\smallskipamount}
\def\medskip{\vskip\medskipamount}
\def\bigskip{\vskip\bigskipamount}

\numberwithin{equation}{section}

\theoremstyle{plain}
\newtheorem{theorem}{Theorem}[section]
\newtheorem{lemma}[theorem]{Lemma}
\newtheorem{proposition}[theorem]{Proposition}
\newtheorem{corollary}[theorem]{Corollary}

\theoremstyle{definition}
\newtheorem{definition}[theorem]{Definition}

\newtheorem{remark}[theorem]{Remark}
\newtheorem{observation}[theorem]{Observation}
\newtheorem{statement}[theorem]{Statement}

\def\operatorname#1{{\mathop{\rm #1}}}
\newcommand{\ord}{\operatorname{ord}}

\def\comp{\leftrightarrow}

\newcommand\implik{{\ \Longrightarrow\ }}

\newcounter{ok}
{\end{list}}

\newcounter{aok}
{\end{list}}

\def\go#1;#2;#3 {\vbox to0pt{\kern-#3\hbox{\kern#2 #1}\vss}\nointerlineskip}

\newcommand{\Mea}{{\text{\rm{}M}}}
\newcommand{\HMea}{{\text{\rm{}HM}}}

\newcommand{\Tea}{{\text{\rm{}T}}}

\newcommand{\C}{\text{\rm{}C}}
\newcommand{\Sh}{\text{\rm{}S}}

\newcommand{\itrm}[1]{\item[{\rm {#1}}]}

 \renewcommand{\le}{\leqslant}

\begin{document}

\title[Triple Representation Theorem]{Triple Representation Theorem for orthocomplete \\
homogeneous effect algebras}

\author[J. Niederle]{Josef~Niederle}
\email{niederle@math.muni.cz}

\author[J. Paseka]{Jan~Paseka} 
\email{paseka@math.muni.cz}
\address{Department of Mathematics and Statistics,
			Faculty of Science,
			Masaryk University,
			{Kotl\'a\v r{}sk\' a\ 2},
			CZ-611~37~Brno, Czech Republic}        	

\def\logoesf{
\begin{tabular}{l l}
\begin{tabular}{c}
{Supported by}\\
\phantom{\huge X}
\end{tabular}& \ \resizebox{8.58cm}{!}{\includegraphics{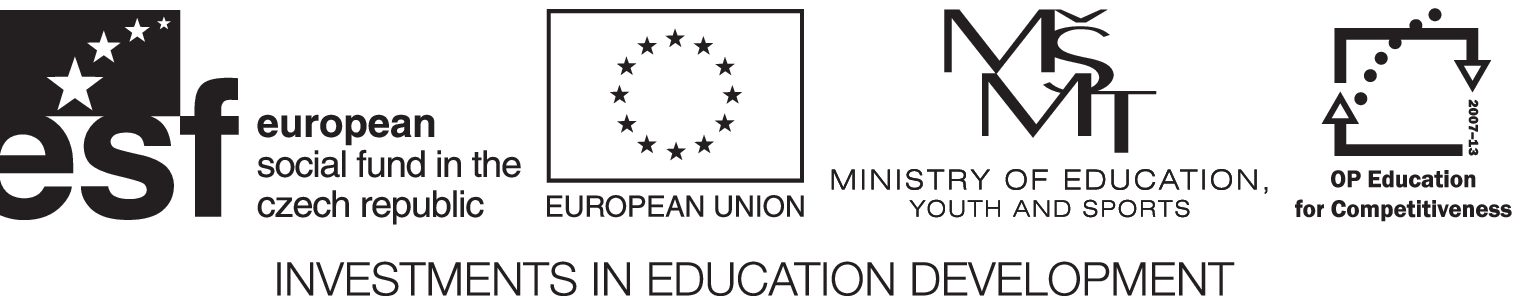}}
\end{tabular}}
\smallskip
\thanks{J. Paseka gratefully acknowledges Financial Support 
of the  Ministry of Education of the Czech Republic
under the project MSM0021622409 and of Masaryk University under the grant 0964/2009. 
Both authors acknowledge the support by ESF Project CZ.1.07/2.3.00/20.0051
Algebraic methods in Quantum Logic of the Masaryk University.\\
\logoesf}




\begin{abstract}
The aim of our paper is twofold. First, 
we thoroughly study the set of meager elements $\Mea(E)$, 
the set of sharp elements $\Sh(E)$ and
the center $\C(E)$ in the setting of
meager-orthocomplete homogeneous effect algebras $E$. 
Second, we prove the Triple Representation Theorem for 
sharply dominating meager-orthocomplete homogeneous effect algebras,
in particular orthocomplete homogeneous effect algebras. 
\end{abstract}


\maketitle

\zmena{\section*{Introduction}}

\label{intro}

Two equivalent quantum structures, D-posets and effect algebras were
introduced in the nineties of the twentieth century. These
were considered as ``unsharp'' generalizations of the
structures which arise in quantum mechanics, in particular, of orthomodular
lattices and MV-algebras. Effect algebras aim to
describe ``unsharp'' event structures in quantum mechanics
in the language of algebra.

Effect algebras are fundamental in investigations of fuzzy probability
theory too. In the fuzzy probability frame, the elements of an effect algebra
represent fuzzy events which are used to construct fuzzy random variables.

The aim of our paper is twofold. First, 
we thoroughly study the set of meager elements $\Mea(E)$, 
the set of sharp elements $\Sh(E)$ and
the center $\C(E)$ in the setting of
meager-orthocomplete homogeneous effect algebras $E$.
Second, in Section~\ref{tretikapitola}  we prove the Triple Representation Theorem,
which was established 
by Jen\v{c}a in \cite{jenca} in the setting of complete lattice effect 
algebras, for sharply dominating meager-orthocomplete
homogeneous effect algebras, in particular orthocomplete
homogeneous effect algebras.

As a by-product of our study we show that an effect algebra $E$ is Archime\-dean if and only if 
the corresponding generalized  
effect algebra $\Mea(E)$ is Archi\-me\-dean and that 
any  homogeneous meager-orthocomplete sharply dominating 
 effect algebra  $E$  can be covered  by Archimedean  Heyting effect algebras  which form
blocks.

\medskip

\section{{Preliminaries} and basic facts}

\zmena{Effect algebras were introduced by Foulis and 
 Bennett (see \cite{FoBe}) 
for modelling unsharp measurements in a Hilbert space. In this case the 
set $\EuScript E(H)$ of effects is the set of all self-adjoint operators $A$ on 
a Hilbert space $H$ between the null operator $0$ and the identity 
operator $1$ and endowed with the partial operation $+$ defined 
iff  $A+B$ is in $\EuScript E(H)$, where $+$ is the usual operator sum. }

\zmena{In general form, an effect algebra is in fact a partial algebra 
with one partial binary operation and two unary operations satisfying 
the following axioms due to Foulis and 
 Bennett.}

\begin{definition}\label{def:EA}{\autori{{\rm{}\cite{ZR62}  }}
{\rm A partial algebra $(E;\oplus,0,1)$ is called an {\em effect algebra} if
$0$, $1$ are two distinct elements, 
called the {\em zero} and the {\em unit}  element,  and $\oplus$ is a partially
defined binary operation called the {\em orthosummation} 
on $E$ which satisfy the following
conditions for any $x,y,z\in E$:
\begin{description}
\item[\rm(Ei)\phantom{ii}] $x\oplus y=y\oplus x$ if $x\oplus y$ is defined,
\item[\rm(Eii)\phantom{i}] $(x\oplus y)\oplus z=x\oplus(y\oplus z)$  if one
side is defined,
\item[\rm(Eiii)] for every $x\in E$ there exists a unique $y\in
E$ such that $x\oplus y=1$ (we put $x'=y$),
\item[\rm(Eiv)\phantom{i}] if $1\oplus x$ is defined then $x=0$.
\end{description}
}%
{\rm{}$(E;\oplus,0,1)$ is  called an {\em orthoalgebra} if 
$x\oplus x$ exists implies that $x= 0$  (see \cite{grerut}).}}
\end{definition}

We often denote the effect algebra $(E;\oplus,0,1)$ briefly by
$E$. On every effect algebra $E$  a partial order
$\le$  and a partial binary operation $\ominus$ can be 
introduced as follows:
\begin{center}
$x\le y$ \mbox{ and } {\autori{$y\ominus x=z$}} \mbox{ iff }$x\oplus z$
\mbox{ is defined and }$x\oplus z=y$\,.
\end{center}

If $E$ with the defined partial order is a lattice (a complete
lattice) then $(E;\oplus,0,1)$ is called a {\em lattice effect
algebra} ({\em a complete lattice effect algebra}).

Mappings from one effect algebra to
another one that preserve units and orthosums are called 
{\em morphisms of effect algebras}, and bijective
morphisms of effect algebras having inverses that are 
morphisms of effect algebras are called 
{\em isomorphisms of effect algebras}.

\begin{definition}\label{subef}{\rm
Let $E$ be an  effect algebra.
Then $Q\subseteq E$ is called a {\em sub-effect algebra} of  $E$ if 
\begin{enumerate}
\item[(i)] $1\in Q$
\item[(ii)] if out of elements $x,y,z\in E$ with $x\oplus y=z$
two are in $Q$, then $x,y,z\in Q$.
\end{enumerate}
If $E$ is a lattice effect algebra and $Q$ is a sub-lattice and a sub-effect
algebra of $E$, then $Q$ is called a {\em sub-lattice effect algebra} of $E$.}
\end{definition}

Note that a sub-effect algebra $Q$ 
(sub-lattice effect algebra $Q$) of an  effect algebra $E$ 
(of a lattice effect algebra $E$) with inherited operation 
$\oplus$ is an  effect algebra (lattice effect algebra) 
in its own right.

\begin{definition}

\rm
{\it{}(1)}: A {\em generalized effect algebra} ($E$; $\oplus$, 0) is a set $E$ with element $0\in E$ and partial
binary operation $\oplus$ satisfying, for any $x,y,z\in E$, conditions{\advance\parindent by 5pt
\begin{enumerate}
\item[(GE1)] $x\oplus y = y\oplus x$ if one side is defined,

\item[(GE2)] $(x\oplus y)\oplus z=x\oplus(y\oplus z)$ if one side is
defined,

\item[(GE3)] if $x\oplus y=x\oplus z$ then $y=z$,

\item[(GE4)] if $x\oplus y=0$ then $x=y=0$,

\item[(GE5)] $x\oplus 0=x$ for all $x\in E$.
\end{enumerate}}
\begin{description}
\item[(2)] A binary relation $\le$ (being a partial order) and 
a partial binary operation $\ominus$ on $E$ can be defined by:
\begin{center}
$x\le y$ \mbox{ and } {{$y\ominus x=z$}} \mbox{ iff }$x\oplus z$
\mbox{ is defined and }$x\oplus z=y$\,.
\end{center} 
\item[(3)] A nonempty subset $Q\subseteq E$ is called a {\em sub-generalized effect algebra} 
of $E$ if out of elements $x,y,z\in E$ with $x\oplus y=z$ at least two are in $Q$ then $x,y,z\in
Q$.  Then $Q$ is a generalized effect algebra in its own right. 
\end{description}
\end{definition}


For an element $x$ of a generalized effect algebra $E$ we write
$\ord(x)=\infty$ if $nx=x\oplus x\oplus\dots\oplus x$ ($n$-times)
exists for every positive integer $n$ and we write $\ord(x)=n_x$
if $n_x$ is the greatest positive integer such that $n_xx$
exists in $E$.  A generalized effect algebra $E$ is {\em Archimedean} if
$\ord(x)<\infty$ for all $x\in E$.

Every effect algebra is a generalized effect algebra.

\begin{definition}
\rm
We say that a finite system $F=(x_k)_{k=1}^n$ of not necessarily
different elements of a generalized effect algebra $E$ is
\zmena{\em orthogonal} if $x_1\oplus x_2\oplus \cdots\oplus
x_n$ (written $\bigoplus\limits_{k=1}^n x_k$ or $\bigoplus F$) exists
in $E$. Here we define $x_1\oplus x_2\oplus \cdots\oplus x_n=
(x_1\oplus x_2\oplus \cdots\oplus x_{n-1})\oplus x_n$ supposing
that $\bigoplus\limits_{k=1}^{n-1}x_k$ is defined and
$(\bigoplus\limits_{k=1}^{n-1}x_k)\oplus x_n$ exists. We also define 
$\bigoplus \emptyset=0$.
An arbitrary system
$G=(x_{\kappa})_{\kappa\in H}$ of not necessarily different
elements of $E$ is called \zmena{\em orthogonal} if $\bigoplus K$
exists for every finite $K\subseteq G$. We say that for a \zmena{orthogonal} 
system $G=(x_{\kappa})_{\kappa\in H}$ the
element $\bigoplus G$ exists iff
$\bigvee\{\bigoplus K
\mid
K\subseteq G$ is finite$\}$ exists in $E$ and then we put
$\bigoplus G=\bigvee\{\bigoplus K\mid K\subseteq G$ is
finite$\}$. We say that $\bigoplus G$ is the {\em orthogonal sum} 
of $G$ and $G$ is {\em orthosummable}. (Here we write $G_1\subseteq G$ iff there is
$H_1\subseteq H$ such that $G_1=(x_{\kappa})_{\kappa\in
H_1}$).
{
We denote $G^\oplus:=\{\bigoplus K\mid K\subseteq G$ is   
finite$\}$. $G$ is called {\em bounded} if there is an upper bound 
of $G^\oplus$.}
\end{definition}

Note that, in any ef\/fect algebra $E$, the following
inf\/inite distributive law holds (see \cite[Proposition 1.8.7]{dvurec}):
\begin{equation}
\tag{IDL}
\Big(\bigvee_{\alpha} c_{\alpha}\Big)\oplus b = \bigvee_{\alpha} (c_{\alpha}\oplus b)
\end{equation}
provided that $\bigvee_{\alpha} c_{\alpha}$ and  $(\bigvee_{\alpha} c_{\alpha})\oplus b$
exist. We then have the following proposition.

\begin{proposition}\label{infasoc}
 Let $E$ be a generalized effect algebra, 
$G_1=(x_{\kappa})_{\kappa\in H_1}$ and $G_2=(x_{\kappa})_{\kappa\in H_2}$  
be orthosummable orthogonal systems in $E$ such that $H_1\cap H_2=\emptyset$ and 
$\bigoplus G_1\oplus \bigoplus G_2$ exists. 
Then $G=(x_{\kappa})_{\kappa\in {H_1\cup H_2}}$ is an orthosummable 
orthogonal system and $\bigoplus G=\bigoplus G_1\oplus \bigoplus G_2$.
\end{proposition}
\begin{proof}  For any finite  
$N_1\subseteq {H_1}$ and any finite  
$N_2\subseteq {H_2}$, we have that the orthosum 
  $\bigoplus_{\kappa\in N_1} x_{\kappa} \oplus 
\bigoplus_{\kappa\in N_2} x_{\kappa}\leq \bigoplus G_1\oplus \bigoplus G_2$ 
exists. Hence $G$ is an orthogonal system.

Evidently, $\bigoplus G_1\oplus \bigoplus G_2\geq \bigoplus F$
for any finite system $F=(x_{\kappa})_{\kappa\in H}$, $H\subseteq {H_1\cup H_2}$ finite. 
Therefore, $\bigoplus G_1\oplus \bigoplus G_2$ is an upper bound of $G^\oplus$. 
Let $z\in E$ be any upper bound  of $G^\oplus$. Then, for any finite  
$N_1\subseteq {H_1}$ and any finite  
$N_2\subseteq {H_2}$, we have that $z\geq \bigoplus_{\kappa\in N_1} x_{\kappa} \oplus 
\bigoplus_{\kappa\in N_2} x_{\kappa}$. Therefore 
 $z \ominus \bigoplus_{\kappa\in N_1} x_{\kappa} \geq 
\bigoplus_{\kappa\in N_2} x_{\kappa}$ and 
$z \ominus \bigoplus_{\kappa\in N_1} x_{\kappa}$ is an upper bound of $G_2^\oplus$. 
This yields that,  for any finite  
$N_1\subseteq {H_1}$, $z \ominus \bigoplus_{\kappa\in N_1} x_{\kappa}\geq \bigoplus G_2$, 
i.e., $z \ominus \bigoplus G_2\geq \bigoplus_{\kappa\in N_1} x_{\kappa}$ .
Hence $z \ominus \bigoplus G_2$ is an upper bound of $G_1^\oplus$. 
Summing up, $z\geq \bigoplus G_1\oplus \bigoplus G_2$ and this gives that 
 $G$ is orthosummable and 
$\bigoplus G=\bigoplus G_1\oplus \bigoplus G_2$.
\end{proof}

\begin{definition}
\rm A generalized effect algebra
$E$ is called \emph{orthocomplete} if  every bounded orthogonal 
system  is orthosummable.
\end{definition}

\begin{observation}\label{semiarch}
Let $E$ be an {orthocomplete} generalized effect algebra, 
$x, y\in E$ such that $ny\leq x$ for every positive integer $n$.
Then $y=0$ .
\end{observation}
\begin{proof} Let $G=(g_n)_{n=1}^{\infty}$, $g_n=y$  for every positive integer $n$. Then, for all $K\subseteq G$ finite, we 
have $\bigoplus K\leq x$, hence $G$ is {bounded} and 
$\bigoplus G$ exists. In virtue of (IDL), 
$\bigoplus G=g_1 \oplus \bigoplus (g_n)_{n=2}^{\infty}= 
g_1 \oplus \bigoplus G$. Therefore $0=g_1=y$.
\end{proof}

Let us remark a well known fact that every orthocomplete 
effect algebra is Archimedean.

{\renewcommand{\labelenumi}{{\normalfont  (\roman{enumi})}}
\begin{definition}
\rm
An element $x$ of an effect algebra $E$ is called  
\begin{enumerate}
\item  \emph{sharp} if $x\wedge x'=0$. The set 
$\Sh(E)=\{x\in E \mid x\wedge x'=0\}$ is called a \emph{set of all sharp elements} 
of $E$ (see \cite{gudder1}).
\item  \emph{principal}, if $y\oplus z\leq x$ for every $y, z\in E$ 
such that $y, z \leq x$ and $y\oplus z$ exists.
\item  \emph{central}, if $x$ and $x'$ are principal and,
for every $y \in E$ there are $y_1, y_2\in E$ such that 
$y_1\leq x, y_2\leq x'$, and $y=y_1\oplus y_2$  (see \cite{GrFoPu}). 
The \emph{center} 
$\C(E)$ of $E$ is the set of all central elements of $E$.
\end{enumerate}
\end{definition}

If $x\in E$ is a principal element, then $x$ is sharp and the interval 
$[0, x]$ is an effect algebra with the greatest element $x$ and the partial 
operation given by restriction of $\oplus$ to $[0, x]$. 

\begin{statement}{\rm \cite[Theorem 5.4]{GrFoPu}}
The center $\C(E)$ of an effect algebra $E$ is a sub-effect 
algebra of $E$ and forms a Boolean algebra. For every
central element $x$ of $E$, $y=(y\wedge x)\oplus (y\wedge x')$ for all 
$y\in E$.  If $x, y\in \C(E)$ are orthogonal, we have 
$x\vee y =  x\oplus y$  and $x\wedge y =  0$. 
\end{statement}

{\renewcommand{\labelenumi}{{\normalfont  (\roman{enumi})}}
\begin{statement}{\rm\cite[Lemma   3.1.]{jencapul}}\label{gejzapulm} Let $E$ be an effect algebra, 
$x, y\in E$ and $c, d\in \C(E)$. Then:
\begin{enumerate}
\item If $x\oplus y$ exists then  $c\wedge (x\oplus y)=(c\wedge x)\oplus (c\wedge y)$.
\item If $c\oplus d$ exists then  $x\wedge (c\oplus d)=(x\wedge c)\oplus (x\wedge d)$.  
\end{enumerate}
\end{statement}

\begin{definition}
\rm
A subset $M$ of a generalized effect algebra $E$ is called 
\emph{internally compatible} (\emph{compatible}) if for every finite 
subset $M_F$ of $M$  there is a finite orthogonal family $(x_1, \dots, x_n)$ 
of elements from $M$ ($E$) such that for every $m\in  M_F$ 
there is a set $A_F\subseteq \{ 1, \dots, n \}$ with 
$m =\bigoplus_{i\in A_F} x_i$. If $\{ x, y \}$ is a compatible set, 
we write $x\comp y$ (see \cite{jenca, Kop2}).
\end{definition}

Evidently, $x\comp y$ iff there are $p, q, r\in E$ such 
that $x=p\oplus q$, $y=q\oplus r$ and $p\oplus q\oplus r$ exists iff 
there are $c, d\in E$ such that $d\leq x\leq c$, $d\leq y\leq c$ and
$c\ominus x=y\ominus d$. 
Moreover, if $x\wedge y$ exists then $x\comp y$ iff 
$x\oplus (y\ominus (x\wedge y))$ exists.

\section{Orthocomplete homogeneous  effect algebras}
\label{orto}

\begin{definition}\label{rdp}{\rm
An effect algebra $E$  satisfies the \emph{Riesz decomposition property}
(or RDP) if, for all $u,v_1, v_2\in E$ such that
$u\leq v_1\oplus v_2$, there are $u_1, u_2$ such that 
$u_1\leq v_1, u_2\leq v_2$ and $u=u_1\oplus u_2$. 
\\
A lattice effect algebra in which RDP holds is called an \emph{MV-effect
algebra}.
\\
An effect algebra $E$ is called \emph{homogeneous} if, 
for all $u,v_1,v_2\in E$ such that $u\leq v_1\oplus v_2\leq u'$, 
there are $u_1,u_2$ such that $u_1\leq v_1$, $u_2\leq v_2$ 
and $u=u_1\oplus u_2$ (see \cite{jenca2001}). }
\end{definition}


\begin{lemma}\label{xshom}
Let $E$ be a homogeneous effect algebra and let 
$u,v_1,v_2\in E$ such that $u\leq v_1\oplus v_2\leq u'$ and $v_1\in\Sh(E)$. 
Then $u\leq v_2$ and $u\wedge v_1=0$.
\end{lemma}
\begin{proof} Since $E$ is homogeneous ,
there are $u_1,u_2$ such that $u_1\leq v_1\leq u'$, $u_2\leq v_2$ 
and $u=u_1\oplus u_2$. Let $w\in E$ such that $w\leq u, v_1$. 
Then $w\leq v_1\wedge v_1'=0$.
Therefore also $u_1\leq u\wedge v_1=0$, i.e. 
$u=u_2\leq v_2$. 
\end{proof}

\begin{statement}\label{gejzasum}{\rm\cite[Proposition 2]{jenca}}
\nopagebreak
\begin{enumerate}
\item[{\rm (i)}]    Every orthoalgebra is homogeneous.
\itrm{(ii)}    Every lattice effect algebra is homogeneous.
\itrm{(iii)}    An effect algebra  $E$ has the Riesz decomposition 
property if and only if  $E$ is homogeneous and compatible.

Let $E$ be a homogeneous effect algebra.
\itrm{(iv)}    A subset $B$ of $E$ is a maximal sub-effect 
algebra of $E$ with the Riesz decomposition property (such  
$B$ is called a {\em block} of  $E$) if and only if  $B$ 
is a maximal internally compatible subset of $E$ containing $1$.
\itrm{(v)} Every finite compatible subset of $E$ is a subset of 
some block. This implies that every homogeneous effect algebra 
is a union of its blocks.
\itrm{(vi)}    $\Sh(E)$ is a sub-effect algebra of $E$.
\itrm{(vii)}    For every block $B$, $\C(B) = \Sh(E)\cap B$.
\itrm{(viii)}    Let $x\in B$, where $B$ is a block of $E$. 
Then $\{ y \in E \mid  y \leq x\ \text{and}\ y \leq x'\}\subseteq  B$.
\end{enumerate}
\end{statement}

Hence the class of homogeneous effect algebras includes orthoalgebras, 
effect algebras satisfying the Riesz decomposition property 
and lattice effect algebras.

An important class of effect algebras was introduced by Gudder in \cite{gudder1} 
and \cite{gudder2}. Fundamental example is the 
standard Hilbert spaces effect algebra ${\mathcal E}({\mathcal H})$.

For an element $x$ of an effect algebra $E$ we denote
$$
\begin{array}{r c  l c l}
\widetilde{x}&=\bigvee_{E}\{s\in \Sh(E) \mid s\leq x\}&%
\phantom{xxxxx}&\text{if it exists and belongs to}\ \Sh(E)\phantom{.}\\
\wwidehat{x}&=\bigwedge_{E}\{s\in \Sh(E) \mid s\geq x\}&\phantom{xxxxx}&%
\text{if it exists and belongs to}\ \Sh(E).
\end{array}
$$

\begin{definition} {\rm (\cite{gudder1},  \cite{gudder2}.) 
An effect algebra $(E, \oplus, 0,
1)$ is called {\em sharply dominating} if for every $x\in E$ there
exists $\wwidehat{x}$, the smallest sharp element  such that $x\leq
\wwidehat{x}$. That is $\wwidehat{x}\in \Sh(E)$ and if $y\in \Sh(E)$ satisfies
$x\leq y$ then $\wwidehat{x}\leq y$. }
\end{definition}

Recall that evidently an effect algebra $E$ is sharply dominating iff 
for every $x\in E$ there exists $\widetilde{x}\in \Sh(E)$ such
that $\widetilde{x}\leq x$ and if $u\in \Sh(E)$
satisfies $u\leq x$ then $u\leq \widetilde{x}$ iff for every $x\in E$ there exist 
a smallest sharp element $\wwidehat{x}$ over $x$ and a greatest sharp 
element $\widetilde{x}$ below $x$.

In what follows set (see \cite{jenca,wujunde})
$$\Mea(E)=\{x\in E \mid\ \text{if}\ v\in \Sh(E)\ \text{satisfies}\ v\leq x\ 
\text{then}\ v=0\}.$$

An element $x\in \Mea(E)$ is called {\em meager}. Moreover, $x\in \Mea(E)$ 
iff $\widetilde{x}=0$. Recall that $x\in \Mea(E)$, $y\in E$, $y\leq x$ implies 
$y\in \Mea(E)$ and $x\ominus y\in \Mea(E)$.

\begin{definition}\label{hypsemea}
\rm Let $E$ be an effect algebra and let 
$$\HMea(E)=\{x\in E \mid\ \text{there is}\ y\in E\ \text{such that}\ x\leq y\ 
\text{and}\ x\leq y'\}.$$
An element $x\in \HMea(E)$ is called {\em hypermeager}.
\end{definition}
Every hypermeager element is meager. Since both $\Mea(E)$ and 
$\HMea(E)$ are downsets of $E$ they form together with the corresponding restriction of the operation $\oplus$ a generalized effect algebra. 

\begin{lemma}\label{ordinffin}
(1) Let \(\operatorname{ord}(y)=\infty\) in \(E\).
\(\{ky\mid k\in\mathbb N\}\subseteq \HMea(E)\).
\\
(2)Let \(\operatorname{ord}(y)=n_y\ne\infty\) in \(E\).
\(\{ky\mid k\in\mathbb N, k\leq \frac{n_y}{2}\}\subseteq \HMea(E)\).
\end{lemma}
\begin{proof}
In either case, \((2k)y\) exists in \(E\).
Therefore \(ky\leq (ky)'\) and consequently \(ky\in \HMea(E)\).
\end{proof}
\begin{proposition}\label{archim}
Let \(E\) be an effect algebra. The following conditions are equivalent.
\begin{enumerate}
\item \(E\) is Archimedean;
\item \(\Mea(E)\) is Archimedean;
\item \(\HMea(E)\) is Archimedean.
\end{enumerate}
\end{proposition}
\begin{proof}
(i) \(\implies\) (ii)
If \(E\) is Archimedean, \(\Mea(E)\) is Archimedean a fortiori.
\\
(ii) \(\implies\) (iii)
If \(\Mea(E)\) is Archimedean, \(\HMea(E)\) is Archimedean a fortiori.
\\
(iii) \(\implies\) (i)
Let \(\HMea(E)\) be Archimedean. Suppose \(\operatorname{ord}(y)=\infty\)
in \(E\) where \(y\ne 0\).
By Lemma~\ref{ordinffin}, \(\{ky\mid k\in \mathbb N\}\subseteq \HMea(E)\),
which contradicts the assumption.
\end{proof}

{\renewcommand{\labelenumi}{{\normalfont  (\roman{enumi})}}
\begin{statement}\label{jmpy2} {\rm\cite[Lemma 2.4]{niepa}} Let 
$E$ be an effect algebra 
in which $\Sh(E)$ is a sub-effect algebra of $E$ and 
let $x\in \Mea(E)$ such that $\wwidehat{x}$  exists. Then 
\begin{enumerate}
\settowidth{\leftmargin}{(iiiii)}
\settowidth{\labelwidth}{(iii)}
\settowidth{\itemindent}{(ii)}
\item $\wwidehat{x}\ominus x\in \Mea(E)$.
\item If $y\in \Mea(E)$ such that $x\oplus y$ %
exists and $x\oplus y=z\in\Sh(E)$ then  $\wwidehat{x}=z$. 
\end{enumerate}
\end{statement}}

{\renewcommand{\labelenumi}{{\normalfont  (\roman{enumi})}}
\begin{statement}\label{jpy2} {\rm\cite[Lemma 2.5]{niepa}} Let $E$ be an effect algebra 
in which $\Sh(E)$ is a sub-effect algebra of $E$ and 
let $x\in E$ such that $\widetilde{x}$  
exists. Then 
${x\ominus \widetilde{x}}\in  \Mea(E)$ and 
$x=\widetilde{x}\oplus {(x\ominus \widetilde{x})}$ is 
the unique decomposition $x = x_S \oplus x_M$, 
where $x_S\in\Sh(E)$ and $x_M \in \Mea(E)$. Moreover, 
$x_S\wedge x_M=0$ and if $E$ is a lattice effect algebra then 
 $x = x_S \vee x_M$. 
\end{statement}}

As proved in \cite{cattaneo}, 
$\Sh(E)$ is always a sub-effect algebra in 
a sharply dominating  effect algebra $E$.

\begin{corollary}\label{gejza}{\rm{}\cite[Proposition 15]{jenca}}
Let $E$ be a sharply dominating  effect algebra. 
Then every $x \in E$ has a
unique decomposition $x = x_S \oplus x_M$, where $x_S\in\Sh(E)$ and $x_M \in \Mea(E)$, 
namely $x=\widetilde{x}\oplus {(x\ominus \widetilde{x})}$.
\end{corollary}


\begin{statement}\label{gejzaoch}{\rm\cite[Corollary 14]{jenca}} Let $E$ be an orthocomplete 
homogeneous effect algebra. Then  $E$ is 
sharply dominating. 
\end{statement}

{
\begin{proposition}\label{modyjem} Let $E$ be a 
homogeneous effect algebra and $v\in E$.
The following conditions are equivalent.
\begin{enumerate}
\item $v\in \Sh(E)$;
\item $y\leq z$ whenever $w, y, z\in E$   such that $v=w\oplus z$,
$y\leq w'$ and $y\leq w$. 
\item $[0, w]\cap [0, w']=[0, w]\cap [0, v\ominus w]$ whenever $w\in E$ 
and $w\leq v$.
\end{enumerate}
\end{proposition}
\begin{proof}
(i) $\implies$ (ii): 
Evidently, there is a block, say $B$, such that it contains the following 
orthogonal system $\{y, w\ominus y, z, 1\ominus v\}$. Hence 
$B$ contains also $w$, $w'$ and $v\in \C(B)$. Since 
$1=w\oplus w'$ we obtain by 
Statement \ref{gejzapulm}, (ii) that 
$v=v\wedge_B w\oplus %
v\wedge_B w'=%
w\oplus %
v\wedge_B w'$. Subtracting $w$ we obtain 
$z=v\wedge_B w'$. Hence
 $y\leq w \leq v$ and $y\leq w'$ yields that $y\leq z$.
 \newline
(ii) $\implies$  (iii):
Clearly, $[0, w]\cap [0, v\ominus w] \subseteq [0, w]\cap [0, w']$. The 
other inclusion is a direct reformulation of (ii).
\newline
(iii) $\implies$  (i)
Let $y\in [0,v]\cap[0,v']$. Then   from (iii) we have that 
$y\in [0,v]\cap[0, v\ominus v]=\{ 0\}$. Immediately, $y=0$ and 
$v$ is sharp.
\end{proof}
}

\begin{corollary}{\rm\cite[Lemma 2.12]{niepa2}}\label{soucethat}
Let $E$ be a homogeneous effect algebra, and $y\in E$ and $w\in \Sh(E)$ for which
$y\leq w$ and $ky$ exists. It holds $ky\leq w$.
\end{corollary}

\begin{corollary}\label{wsmodyjem} Let $E$ be a 
homogeneous effect algebra and let $x, y\in E$  such that $\widehat{x}$ exists, 
$y\leq (\widehat{x}\ominus x)'$ and $y\leq \widehat{x}\ominus x$. 
Then $y\leq x$. 
\end{corollary}
\begin{proof} It is enough to put in Proposition \ref{modyjem} 
$v=\widehat{x}$, $w=\widehat{x}\ominus x$ and $z=x$. 
\end{proof}

{
\begin{statement}{\rm\cite[Lemma  1.16.]{niepa2}}\label{hatrozdilu}
Let $E$ be a sharply dominating effect algebra and let $x\in E$. Then 
\begin{equation*}
\wwidehat{x\ominus\widetilde{x}}=
\wwidehat{\wwidehat{x}\ominus x}=
\wwidehat{x}\ominus \widetilde{x} .
\end{equation*}
\end{statement}

{\renewcommand{\labelenumi}{{\normalfont  (\roman{enumi})}}
\begin{lemma}\label{suplem}
Let $E$ be an effect algebra and let $x\in E$. Then 
\begin{enumerate}
\settowidth{\leftmargin}{(iiiii)}
\settowidth{\labelwidth}{(iii)}
\settowidth{\itemindent}{(ii)}
\item If\/ $\widehat{x}$ exists then \(\widetilde{(x')}\) exists and
\begin{equation*}
\wwidehat{x}\ominus x=%
 x'\ominus (\wwidehat{x})'=x'\ominus\widetilde{(x')}.
\end{equation*}
\item If\/ $\widetilde{x}$ exists then \(\wwidehat{(x')}\) exists and
\begin{equation*}
x \ominus \widetilde{x}=(\widetilde{x})'\ominus x'=\wwidehat{(x')}\ominus x'.
\end{equation*}
\end{enumerate}
\end{lemma}}

\begin{proof}
Transparent.
\end{proof}
}

{\renewcommand{\labelenumi}{{\normalfont  (\roman{enumi})}}
\begin{lemma}\label{dusuplem}
Let $E$ be an effect algebra and let $x, y\in E$.
\begin{enumerate}
\settowidth{\leftmargin}{(iiiii)}
\settowidth{\labelwidth}{(iii)}
\settowidth{\itemindent}{(ii)}
\item If\/ $\widehat{x}$ exists then 
\begin{equation*}
y\leq \wwidehat{x}\ominus x%
\qquad \text{if and only if}\qquad y\leq x'\quad \text{and}\quad \wwidehat{x\oplus y}=\wwidehat{x}.
\end{equation*}
\item If\/ $\widetilde{x}$ exists then 
\begin{equation*}
y\leq x \ominus \widetilde{x} %
\qquad \text{if and only if}\qquad y\leq x \quad \text{and}\quad \wwidetilde{x\ominus y}=\wwidetilde{x}.
\end{equation*}
\end{enumerate}
\end{lemma}}

\begin{proof}
Transparent.
\end{proof}
}

{\renewcommand{\labelenumi}{{\normalfont  (\roman{enumi})}}
\begin{lemma}\label{xssuplem}
Let $E$ be a homogeneous effect algebra and let $x\in E$. 
\begin{enumerate}
\settowidth{\leftmargin}{(iiiii)}
\settowidth{\labelwidth}{(iii)}
\settowidth{\itemindent}{(ii)}
\item If\/ $\widetilde{x}$ exists then 
\begin{equation*}
[0,x]\cap [0, x']=[0,x\ominus \widetilde{x}]\cap [0, x']=%
[0,x\ominus \widetilde{x}]\cap [0, (x\ominus \widetilde{x})'].
\end{equation*}
\item If\/ $\widehat{x}$ exists then 
\begin{equation*}
[0,x]\cap [0, x']=[0,x]\cap [0, \wwidehat{x}\ominus x]=%
[0,(\wwidehat{x}\ominus x)']\cap [0, \wwidehat{x}\ominus x].
\end{equation*}
\item If both\/ $\widetilde{x}$ and $\widehat{x}$ exist then 
\begin{equation*}
[0,x]\cap [0, x']=[0,x\ominus \widetilde{x}]\cap [0, \wwidehat{x}\ominus x].
\end{equation*}
\end{enumerate}
\end{lemma}}

\begin{proof} (i): Clearly,
$[0,x]\cap [0, x']\supseteq [0,x\ominus \widetilde{x}]\cap [0, x']$ 
and $[0,x\ominus \widetilde{x}]\cap [0, x']\subseteq %
[0,x\ominus \widetilde{x}]\cap 
[0, x'\oplus \widetilde{x}]=[0,x\ominus \widetilde{x}]\cap [0, (x\ominus \widetilde{x})']$. 
Assume that $y\in [0,x]\cap [0, x']$. Then 
by Proposition \ref{modyjem} applied to $\wwidehat{x'}$ and $x'$
we obtain that 
$y\leq \wwidehat{x'}\ominus x'=x\ominus \widetilde{x}$. Hence 
$[0,x]\cap [0, x']=[0,x\ominus \widetilde{x}]\cap [0, x']$.

Now, assume that $z\in [0,x\ominus \widetilde{x}]%
\cap [0, (x\ominus \widetilde{x})']$. Then 
$z\leq (x\ominus \widetilde{x})'=x'\oplus \widetilde{x}\leq z'$. 
Therefore $z=z_1\oplus z_2$, $z_1\leq x'$ and 
$z_2\leq \widetilde{x}\leq z'\leq z_2'$. Hence 
$z_2\leq \widetilde{x}\wedge \widetilde{x}'=0$. This yields 
that $z=z_1\leq x'$, i.e. 
$[0,x\ominus \widetilde{x}]\cap [0, x']=%
[0,x\ominus \widetilde{x}]\cap [0, (x\ominus \widetilde{x})']$.

(ii): It follows by interchanging $x$ and $x'$.

(iii): We have 
\begin{equation*}
[0,x]\cap [0, x']=%
[0,x\ominus \widetilde{x}]\cap [0, (x\ominus \widetilde{x})']%
\cap [0,(\wwidehat{x}\ominus x)']\cap [0, \wwidehat{x}\ominus x].
\end{equation*}
Since $x\ominus \widetilde{x}\leq (\wwidehat{x}\ominus x)'$ 
and $\wwidehat{x}\ominus x\leq (x\ominus \widetilde{x})'$ we 
are finished.
\end{proof}

\begin{lemma}\label{exssuplem}
Let $E$ be a  homogeneous effect algebra and let $x, x_1, \dots, x_n\in E$ be such that 
$\bigoplus_{i=1}^{n} x_i\leq x$, $\widetilde{x}$ exists and, for all $i=1, \dots, n$, $x\leq x_i'$. 
Then $\bigoplus_{i=1}^{n} x_i\leq x\ominus \widetilde{x}$ and therefore
$\widetilde{x}=\wwidetilde{x\ominus \bigoplus_{i=1}^{n} x_i}$.
\end{lemma}
\begin{proof} For $n=0$, the statement trivially holds. Assume that 
the statement is satisfied for some $n$. Then 
$\bigoplus_{i=1}^{n+1} x_i\leq x=\widetilde{x}%
\oplus (x\ominus \widetilde{x})$ and, for all $i=1,
\dots, n+1$, $x_i\leq x'$ yield that 
$\bigoplus_{i=1}^{n} x_i\leq x\ominus \widetilde{x}$, and clearly
$x_{n+1}\leq
x\ominus \bigoplus_{i=1}^{n} x_i$. %
Since also $x_{n+1}\leq x'\leq (x\ominus \bigoplus_{i=1}^{n} x_i)'$,
further \(x_{n+1}\leq \wwidetilde{x}\oplus ((x\ominus\wwidetilde{x})
\ominus\bigoplus_{i=1}^{n} x_i)\leq x_{n+1}'\). By 
Lemma \ref{xshom} we obtain
$x_{n+1}\leq (x\ominus \wwidetilde{x})\ominus \bigoplus_{i=1}^{n} x_i$. 
This yields that $\bigoplus_{i=1}^{n+1} x_i\leq x\ominus \widetilde{x}$.
\end{proof}


{
}

\section{Blocks and orthogonal sums of hypermeager elements in
meager-orthocomplete effect algebras}
\label{podrobnosti}

\begin{definition} \rm \label{meagerocm}
An effect algebra \(E\) is {\em meager-orthocomplete} if \(\Mea(E)\)
is an orthocomplete generalized effect algebra. For a bounded 
orthogonal family $(v_i)_{i\in I}$ in $\Mea(E)$ we shall denote by 
$\bigoplus_{i\in I}^{\Mea(E)} v_i$ the orthogonal 
sum of $(v_i)_{i\in I}$ calculated in \(\Mea(E)\).
\end{definition}

\begin{observation}
Every orthocomplete effect algebra is meager-orthocomplete
and sharply dominating.
\end{observation}

\begin{proposition}\label{xhcmea}
Let $E$ be a homogeneous  meager-orthocomplete
 effect algebra. 
Let $(v_i)_{i\in I}$ be an orthogonal family such that 
$v=\bigoplus_{i\in I}^{\Mea(E)} v_i$ exists, $v\in \Mea(E)$ 
and $u\in E$ be such that $u\leq v\leq u'$. 
Then there is an orthogonal family $(u_i)_{i\in I}$ such that 
$u=\bigoplus_{i\in I}^{\Mea(E)} u_i$ exists and $u_i\leq v_i$ for all 
$i\in I$.
\end{proposition}
\begin{proof} Let us put $X=\{(x, i)\in E\times I \mid 
i\in I, x\leq v_i, x\leq u\}$. We say that a subset $Y\subseteq X$ is 
$u${\em{}-good} if 
{\renewcommand{\labelenumi}{(\roman{enumi})}
\begin{enumerate}
\item For all $f, g\in Y$, 
$\pi_2(f)=\pi_2(g)$ implies $f=g$.
\item
\(\bigoplus_{y\in Y}^{\Mea(E)} \pi_1(y)\leq u\).
\item   
$u\ominus \bigoplus_{y\in Y}^{\Mea(E)} \pi_1(y)\leq v \ominus 
\bigoplus_{y\in Y}^{\Mea(E)} v_{\pi_2(y)}$. 
\end{enumerate}}
Let us denote 
$u-X$ the system of all $u$-good subsets of $X$ ordered by inclusion. Then 
$\emptyset\in u-X$ and the union of any chain in $u-X$ is again in $u-X$ in
virtue of (IDL). 
Hence by Zorn's lemma there is a maximal element, say $Z$, in $u-X$.

Let us show that $\pi_2(Z)=I$. Assume the contrary. Then there is $j\in I$ 
such that $j\not\in \pi_2(Z)$. 
Then 
$$\begin{array}{r@{\,\,}c@{\,\,} l}
u\ominus \bigoplus_{y\in Z}^{\Mea(E)} \pi_1(y)\leq v \ominus 
\bigoplus_{y\in Z}^{\Mea(E)} v_{\pi_2(y)}&=&v_j\oplus ((v \ominus 
\bigoplus_{y\in Z}^{\Mea(E)} v_{\pi_2(y)})\ominus v_j)\\
&\leq& u'\leq 
(u\ominus \bigoplus_{y\in Z}^{\Mea(E)} \pi_1(y))'.\phantom{\text{\Huge L}}
\end{array}
$$
Since $E$  is homogeneous we get that 
$u\ominus \bigoplus_{y\in Z}^{\Mea(E)} \pi_1(y)=u_j\oplus x$ such that 
$u_j\leq v_j$ and $x\leq (v \ominus %
\bigoplus_{y\in Z}^{\Mea(E)} v_{\pi_2(y)})\ominus v_j$. 
Hence $(u\ominus \bigoplus_{y\in Z}^{\Mea(E)} \pi_1(y))\ominus u_j=x 
\leq (v \ominus \bigoplus_{y\in Z}^{\Mea(E)} v_{\pi_2(y)})\ominus v_j$.
This yields that 
the set $Z\cup \{(u_j, j)\}$ is $u$-good, a contradiction with 
the maximality of $Z$.

Therefore $\pi_2(Z)=I$. But this yields 
$u\ominus \bigoplus_{y\in Z}^{\Mea(E)} \pi_1(y)\leq v \ominus 
\bigoplus_{y\in Z}^{\Mea(E)} v_{\pi_2(y)}=0$. Let us put 
$u_{\pi_2(y)}=\pi_1(y)$ for all $y\in Z$. Hence 
$u\ominus \bigoplus_{i\in I}^{\Mea(E)} u_i=0$, i.e.,   
$u= \bigoplus_{i\in I}^{\Mea(E)} u_i$.
\end{proof}

\begin{proposition}\label{ssxhcmea}
Let $E$ be a  homogeneous  meager-orthocomplete sharply dominating 
 effect algebra. 
Let $v_1, v_2\in E$ such that $v_1\leq v_2'$.
Let $(u_i)_{i\in I}$ be an orthogonal family such that 
$u=\bigoplus_{i\in I}^{\Mea(E)} u_i\in \Mea(E)$ exists, 
$u\leq v_1\oplus v_2$ and, for all $i\in I$, 
$v_1\oplus v_2\leq u_i'$. 
Then there are orthogonal families $(v^{1}_i)_{i\in I}$ and 
$(v^{2}_i)_{i\in I}$ such that 
$w_1=\bigoplus_{i\in I}^{\Mea(E)} v^{1}_i\leq v_1$ exists,  
$w_2=\bigoplus_{i\in I}^{\Mea(E)} v^{2}_i\leq v_2$ exists, 
$u=w_1\oplus w_2$
and, for all $i\in I$, 
$v^{1}_i\oplus v^{2}_i= u_i$, $v_1\leq {v^{1}_i}'$ and 
$v_2\leq {v^{2}_i}'$.
\end{proposition}
\begin{proof} Let us put $X_{1,2}=\{(x_1, x_2, i)\in E^{2}\times I \mid 
i\in I, x_1\leq v_1, x_2\leq v_2\}$. We say that a subset $Y\subseteq X_{1,2}$ is 
$u_{1,2}${\em{}-good} if 
{\renewcommand{\labelenumi}{(\roman{enumi})}
\begin{enumerate}
\item For all $f, g\in Y$, 
$\pi_3(f)=\pi_3(g)$ implies $f=g$,
\item  $\pi_1(y)\oplus \pi_2(y)= u_{\pi_3(y)}$ for all $y\in Y$,
\item $\bigoplus_{y\in Y}^{\Mea(E)} \pi_1(y)\leq v_1$,
\item $\bigoplus_{y\in Y}^{\Mea(E)} \pi_2(y)\leq v_2$.
\end{enumerate}}
Let us denote $u-X_{1,2}$ the system of 
all $u_{1,2}$-good subsets of $X_{1,2}$ ordered by inclusion. Then 
$\emptyset\in u-X_{1,2}$ and the union 
$Y=\bigcup\{Y_{\alpha}, \alpha\in \Lambda\}$ of a chain of 
$u_{1,2}${{}-good} sets $Y_{\alpha}, \alpha\in \Lambda$ in 
$u-X_{1,2}$ is again in $u-X_{1,2}$. Namely, 
the conditions (i) and (ii) are obviously satisfied. Let us 
check the condition (iii). Let $F\subseteq Y$ be a finite subset 
of $Y$. Then there is $\alpha_0\in \Lambda$ such that 
$F\subseteq Y_{\alpha_0}$. Hence 
$\bigoplus_{y\in F} \pi_1(y)=%
\bigoplus_{y\in F}^{\Mea(E)} \pi_1(y) \leq v_1 \leq v_1\oplus v_2\leq u_{\pi_3(y)}'\leq \pi_1(y)'$ 
for all ${y\in F}$. 
By Lemma \ref{exssuplem} we get that 
$\bigoplus_{y\in F} \pi_1(y)\leq v_1 \ominus \widetilde{v_1}\in \Mea(E)$.  
Since $\Mea(E)$  is an orthocomplete generalized effect algebra we have that  
$\bigoplus_{y\in Y}^{\Mea(E)} \pi_1(y)$ exists in $\Mea(E)$ and  therefore 
$\bigoplus_{y\in Y}^{\Mea(E)} \pi_1(y)\leq v_1 \ominus \widetilde{v_1}\leq v_1$.  
The condition (iv) follows by similar considerations.

By Zorn's lemma there is a maximal element, say $Z$, in $u-X_{1,2}$.

Let us show that $\pi_3(Z)=I$. Assume the contrary. Then there is $j\in I$ 
such that $j\not\in \pi_3(Z)$. 
Therefore by a successive application of Proposition \ref{infasoc} 
$$
\begin{array}{r c l}
u&=&\bigoplus_{i\in I}^{\Mea(E)} u_i=%
\bigoplus_{y\in Z}^{\Mea(E)} u_{\pi_3(y)} \oplus 
\bigoplus_{k\in I\setminus \pi_3(Z)}^{\Mea(E)} u_k\\
&=&%
\bigoplus_{y\in Z}^{\Mea(E)} (\pi_1(y)\oplus \pi_2(y)) \oplus 
\bigoplus_{k\in I\setminus \pi_3(Z)}^{\Mea(E)} u_k%
\phantom{\text{Huge \bf I}}\\
&=&%
\bigoplus_{y\in Z}^{\Mea(E)} \pi_1(y)\oplus \bigoplus_{y\in Z}^{\Mea(E)} \pi_2(y) \oplus 
\bigoplus_{k\in I\setminus \pi_3(Z)}^{\Mea(E)} u_k\leq v_1\oplus v_2.%
\phantom{\text{Huge \bf I}}\\
\end{array}
$$

The last inequality yields 
$$
u_j\leq (v_1\ominus \bigoplus_{y\in Z}^{\Mea(E)} \pi_1(y))%
\oplus (v_2\ominus\bigoplus_{y\in Z}^{\Mea(E)} \pi_2(y))\leq v_1\oplus v_2\leq u_j'.
$$
Since $E$  is homogeneous we get that there are 
$v^{1}_j, v^{2}_j\in E$ such that 
$v^{1}_j\oplus v^{2}_j= u_j$, 
$v^{1}_j\leq v_1\ominus \bigoplus_{y\in Z}^{\Mea(E)} \pi_1(y)$, 
$v^{2}_j\leq v_2\ominus\bigoplus_{y\in Z}^{\Mea(E)} \pi_2(y)$.

This yields that 
the set $Z\cup \{(v^{1}_j, v^{2}_j, j)\}$ is 
$u_{1,2}$-good, a contradiction with 
the maximality of $Z$.

Therefore $\pi_3(Z)=I$. For any $i\in I$ there is a unique 
$y_i\in Z$ such that $\pi_3(y_i)=i$ and, conversely, for any $y\in Z$ there is a unique 
$i_{y}\in I$ such that $\pi_3(y)=i_{y}$. Let us put 
$v^{1}_i=\pi_1(y_i)$, $v^{2}_i=\pi_2(y_i)$, 
$w_1=\bigoplus_{y\in Z}^{\Mea(E)} \pi_1(y)$ and 
$w_2=\bigoplus_{y\in Z}^{\Mea(E)} \pi_2(y)$. Since $Z$ is 
$u_{1,2}$-good we have that 
$w_1\leq v_1$, $w_2\leq v_2$ and 
$w_1\oplus w_2= %
\bigoplus_{y\in Z}^{\Mea(E)} \pi_1(y)\oplus \bigoplus_{y\in Z}^{\Mea(E)} \pi_2(y)=%
\bigoplus_{i\in I}^{\Mea(E)} u_i.$

Moreover, for all $i\in I$, 
 $v_1\leq v_1\oplus v_2 \leq u_i'\leq {v^{1}_i}'$ 
 and 
 $v_2\leq v_1\oplus v_2 \leq u_i'\leq {v^{1}_2}'$.
\end{proof}

\begin{corollary}\label{corssxhcmea}
Let $E$ be a   homogeneous  meager-orthocomplete sharply dominating
 effect algebra. 
Let $v_1, \dots, v_k\in E$ be an orthogonal family.
Let $(u_i)_{i\in I}$ be an orthogonal family such that 
$u=\bigoplus_{i\in I}^{\Mea(E)} u_i\in \Mea(E)$ exists, 
$u\leq v_1\oplus \dots \oplus v_k$ and, for all $i\in I$, 
$v_1\oplus \dots \oplus v_k\leq u_i'$. 
Then there are orthogonal families $(v^{1}_i)_{i\in I}$, $\dots$,
$(v^{k}_i)_{i\in I}$ such that 
$w_1=\bigoplus_{i\in I}^{\Mea(E)} v^{1}_i\leq v_1$ exists,  $\dots$,
$w_k=\bigoplus_{i\in I}^{\Mea(E)} v^{k}_i\leq v_2$ exists, $w_1\leq v_1$, 
$\dots$, $w_k\leq v_k$, $u=w_1\oplus \dots \oplus w_k$
and, for all $i\in I$, 
$v^{1}_i\oplus \dots \oplus v^{k}_i= u_i$, 
$v_1\leq {v^{1}_i}'$, $\dots$, 
$v_k\leq {v^{k}_i}'$.
\end{corollary}
\begin{proof} Straightforward induction with respect to $k$.
\end{proof}

\begin{corollary}\label{secorssxhcmea}
Let $E$ be a   homogeneous  meager-orthocomplete sharply dominating
 effect algebra. 
Let $v, u\in \Mea(E)$, $v\leq u$.
Let $(u_i)_{i\in I}$ be an orthogonal family such that 
$u=\bigoplus_{i\in I}^{\Mea(E)} u_i\in \Mea(E)$ exists 
and, for all $i\in I$, 
$v \oplus (u\ominus v)=u\leq u_i'$. 
Then there are orthogonal families $(v^{1}_i)_{i\in I}$,  
$(v^{2}_i)_{i\in I}$ such that 
$v=\bigoplus_{i\in I}^{\Mea(E)} v^{1}_i$ exists, 
$u\ominus v=\bigoplus_{i\in I}^{\Mea(E)} v^{2}_i$ exists
and, for all $i\in I$, 
$v^{1}_i\oplus v^{2}_i= u_i$, 
$v\leq {v^{1}_i}'$ and $u\ominus v\leq {v^{k}_i}'$.
\end{corollary}
\begin{proof} From Proposition \ref{ssxhcmea} we know that 
there are orthogonal families $(v^{1}_i)_{i\in I}$ and 
$(v^{2}_i)_{i\in I}$ such that 
$w_1=\bigoplus_{i\in I}^{\Mea(E)} v^{1}_i\leq v$ exists,  
$w_2=\bigoplus_{i\in I}^{\Mea(E)} v^{2}_i\leq u\ominus v$ exists, $u=w_1\oplus w_2$
and, for all $i\in I$, 
$v^{1}_i\oplus v^{2}_i= u_i$, $v\leq {v^{1}_i}'$ and 
$u\ominus v\leq {v^{2}_i}'$. We have 
$u=v\oplus (u\ominus v)= w_1 \oplus (v\ominus w_1)\oplus %
w_2 \oplus ((u\ominus v)\ominus w_2)= u\oplus %
(v\ominus w_1) \oplus ((u\ominus v)\ominus w_2)$. 
Therefore $0=v\ominus w_1=(u\ominus v)\ominus w_2$. This yields 
$v=w_1$ and $u\ominus v=w_2$.
\end{proof}

Let $E$ be an meager-orthocomplete
 effect algebra. 
Let $u\in E$. We put 
$\vartheta(u)=\{w\in E \mid w=v \ \text{or}\ w=u\ominus v, %
(u_i)_{i\in I}$ is an orthogonal family such that 
$v=\bigoplus_{i\in I}^{\Mea(E)} u_i\in \Mea(E)$ exists, $v\leq u$ 
and, for all $i\in I$, 
$u\leq u_i'\}$. Clearly, $\{ 0, u\}\subseteq \vartheta(u)$. 
Recall that, for sharply dominating  and homogeneous  $E$ 
and any element 
$v\in \vartheta(u)$ of the above form $\bigoplus_{i\in I}^{\Mea(E)} u_i$,  
we have by Corollary \ref{secorssxhcmea} 
$[0, v]\cup [u\ominus v, u]\subseteq \vartheta(u)$. 
Note also that, for $u\in \Sh(E)$, we obtain 
that $\vartheta(u)=\{ 0, u\}$.
Therefore, we have a map $\Theta:2^E \to 2^E$ defined by 
$\Theta(A)=\bigcup\{\vartheta(u) \mid u\in A\}$ for all 
$A\subseteq E$. The above considerations yield that 
$A\cup \{0\}\subseteq \Theta(A)$.

As in \cite[Theorem 8]{jenca}, for any set $A\subseteq E$, 
$\sigma(A)$ is the smallest superset of $A$ closed with respect to
$\Theta$. Clearly,  $\sigma(A)=\bigcup_{n=0}^{\infty} A_n$,  
where $A_n$ are subsets of $E$ given by the
rules $A_0 = A$, $A_{n + 1} = \Theta(A_n)$.

\begin{lemma}\label{uya}
Let $E$ be a homogeneous meager-orthocomplete
 effect algebra. Let 
$A\subseteq E$ be an internally compatible subset of $E$. 
Then $\sigma(A)$ is internally compatible. 
\end{lemma}
\begin{proof} The proof goes literally the same way 
as in \cite[Theorem 8]{jenca}. Hence we omit it.
\end{proof}

\begin{corollary}\label{blockua} Let $E$ be 
a  homogeneous meager-orthocomplete sharply dominating 
 effect algebra. 
For every block
$B$ of $E$, $\sigma(B)=\Theta(B) = B$.
\end{corollary}

\begin{proposition}\label{cduya}
Let $E$ be a 
homogeneous meager-orthocomplete
 effect algebra, let $x\in \Mea(E)$. Let 
$(x_i)_{i\in I}$ be a maximal orthogonal family such that, 
for all $i \in I$, $x_i \leq x'$ and, for all $F\subseteq I$ finite, 
$\bigoplus_{i\in F} x_i\leq x$.  Then $x=\bigoplus_{i\in I}^{\Mea(E)} x_i$. 
\end{proposition}
\begin{proof} Since $\Mea(E)$ 
is an orthocomplete generalized effect algebra we have that 
$\bigoplus_{i\in I}^{\Mea(E)} x_i$ exists. As in \cite[Theorem 13]{jenca} 
we will prove that $x\ominus \bigoplus_{i\in I}^{\Mea(E)} x_i\in \Sh(E)$. 
Let $r\in E$ be such that 
$$
r\leq x\ominus %
\bigoplus_{i\in I}^{\Mea(E)} x_i, (x\ominus \bigoplus_{i\in I}^{\Mea(E)} x_i)'.
$$
We get that 
$$
r\leq (x\ominus \bigoplus_{i\in I}^{\Mea(E)} x_i)'=%
x'\oplus \bigoplus_{i\in I}^{\Mea(E)} x_i \leq r'.
$$
Therefore there are $r_1, r_2\in E$ such that $r=r_1\oplus r_2$, 
$r_1\leq x'$, $r_2\leq \bigoplus_{i\in I}^{\Mea(E)} x_i $. We have also that 
$r_1\oplus \bigoplus_{i\in I}^{\Mea(E)} x_i\leq x$, i.e. $r_1=0$  and $r_2=r$ by maximality 
of  $(x_i)_{i\in I}$. Then $r\leq \bigoplus_{i\in I}^{\Mea(E)} x_i \leq r'$. 
Proposition \ref{xhcmea} yields that there is an orthogonal family $(u_i)_{i\in I}$ such that 
$r=\bigoplus_{i\in I}^{\Mea(E)} u_i$ exists and $u_i\leq x_i$ for all 
$i\in I$. Hence, for all $j\in I$, 
$$
u_j\leq r\leq x\ominus \bigoplus_{i\in I}^{\Mea(E)} x_i, \ \text{i.e.,}\ 
u_j\oplus \bigoplus_{i\in I}^{\Mea(E)} x_i\leq x\ \text{and}\ u_j\leq x_j\leq x'.
$$
Thus again by maximality of  $(x_i)_{i\in I}$ we get that $u_j=0$. It follows that $r=0$ 
and $x\ominus \bigoplus_{i\in I}^{\Mea(E)} x_i\in \Sh(E)\cap {\Mea(E)}=\{0\}$, i.e. 
$x= \bigoplus_{i\in I}^{\Mea(E)} x_i$. 
\end{proof}

\begin{proposition}\label{corcduya}
Let $E$ be a  
homogeneous meager-orthocomplete sharply dominating
 effect algebra and let $x\in E$.
Let 
$(x_i)_{i\in I}$ be a maximal orthogonal family such that, 
for all $i \in I$, $x_i \leq x'$ and, for all $F\subseteq I$ finite, 
$\bigoplus_{i\in F} x_i\leq x$.  Then 
$x=\widetilde{x}\oplus\bigoplus_{i\in I}^{\Mea(E)} x_i$. 
Moreover, for every block $B$ of $E$, $x \in B$ implies 
that $[\widetilde{x}, x]\subseteq B$ and 
$[x, \widehat{x}]\subseteq B$.
\end{proposition}
\begin{proof} Note first that by Lemma \ref{xssuplem} and Lemma \ref{exssuplem} 
$(x_i)_{i\in I}$ is a maximal orthogonal family such that, 
for all $i \in I$, $x_i \leq (x\ominus \widetilde{x})'$ and, for all $F\subseteq I$ finite, 
$\bigoplus_{i\in F} x_i\leq x\ominus \widetilde{x}$. From Proposition \ref{cduya} we get 
that $\bigoplus_{i\in I}^{\Mea(E)} x_i$ exists and 
$\bigoplus_{i\in I}^{\Mea(E)} x_i=x\ominus \widetilde{x}$. 
Let  $B$ be a block of $E$, $x \in B$. Since 
$\bigoplus_{i\in I}^{\Mea(E)} x_i\in \vartheta({x})\subseteq B$ 
we get by Corollary \ref{secorssxhcmea} 
that $[0, x\ominus \widetilde{x}]\subseteq B$. Therefore $\widetilde{x}\in B$ and 
$[\widetilde{x}, x]=\widetilde{x}\oplus [0, x\ominus \widetilde{x}]\subseteq B$. 
Following the same reasonings for $x'$ we get that 
$[\widetilde{x'}, x']\subseteq B$. Hence also $
[x, \widehat{x}]=[x, \widetilde{x'}']\subseteq B$.
\end{proof}

\begin{corollary}\label{duscduya}
Let $E$ be a  homogeneous meager-orthocomplete
 sharply dominating effect algebra, let $x\in \Mea(E)$.
Then, for every block $B$ of $E$, $x \in B$ implies 
that $[0, x] \subseteq B$ and, moreover, $\Mea(B)\subseteq \Mea(E)$.
\end{corollary}
\begin{proof} Since $x\in \Mea(E)$ we get by 
Corollary \ref{secorssxhcmea} and 
Proposition \ref{cduya} that 
$[0, x]\subseteq \vartheta(x)\subseteq B$.

Now, let $y\in \Mea(B)\subseteq B$, $z\in\Sh(E)$, $z\leq y$. 
Then $z\leq \widetilde{y}\in B$ and $\widetilde{y}\in\Sh(B)$. 
Hence $\widetilde{y}=0$. This yields $z=0$.
\end{proof}

As in \cite[Proposition 16]{jenca} we have that 

\begin{proposition}\label{ocmdcduya}
Let $E$ be a  homogeneous meager-orthocomplete
 sharply dominating  effect algebra, let 
$x \in \Mea(E)$. Then $[0, x]$ is a complete MV-effect algebra.
\end{proposition}
\begin{proof} Let $B$ be a block containing $x$. Since 
$[0, x]\subseteq B$ by Corollary \ref{duscduya} 
and $[0, x]$ is an orthocomplete effect algebra we obtain that 
$[0, x]$ is an orthocomplete effect algebra  satisfying the Riesz decomposition property. From  
\cite[Theorem 4.10]{jencapulquo} we get that 
$[0, x]$ is a lattice and hence a complete MV-effect algebra.
\end{proof}

Recall that  Proposition \ref{ocmdcduya} immediately yields 
(using the same considerations as in 
\cite[Proposition 19]{jenca}) that 
an orthocomplete generalized effect algebra of meager 
elements of a sharply dominating  homogeneous effect algebra  is a  commutative BCK-algebra with the relative cancellation property. Hence, by the result of J.~C\={\i}rulis  (see \cite{cirulis}) it 
is the dual of a weak implication algebra introduced in \cite{chhak}.

\begin{proposition} \label{minimax} 
Let $E$ be a  meager-orthocomplete
 effect algebra,
and let $y,z\in \Mea(E)$.
Every lower bound of $y,z$ is below a maximal one.
\end{proposition}
\begin{proof}
Let $w$ be a lower bound of $y,z$. There exists a maximal orthogonal multiset
$A$ 
containing $w$ in which $0$ occurs uniquely
and for which $y,z$ are upper bounds of $A^\oplus$. 
Indeed, the multiset union of any maximal chain of such multisets is again
such multiset. 

Since $\Mea(E)$ is orthocomplete any non-zero 
element of $A$ has finite multiplicity. 
Again by orthocompleteness, 
there exists a smallest upper bound $u$ of $A^\oplus$
below $y,z$ and hence a lower bound of $y,z$ above $w$.
Let $v$ be an arbitrary lower bound of $y,z$ above $u$.
If $u<v$, the multiset sum $A\uplus \{v\ominus u\}$ is an orthogonal
multiset satisfying all requirements which is properly larger than $A$.
Hence $u=v$ is a maximal lower bound of $y,z$ over $w$.
\end{proof}

The proof of the following Proposition follows the proof 
from {\cite[Proposition 17]{jenca}}.

\begin{proposition} \label{dusminimax} 
Let $E$ be a  homogeneous meager-orthocomplete
 sharply dominating effect algebra. Then 
$\Mea(E)$  is a meet semilattice.
\end{proposition}
\begin{proof}  Let $x, y \in \Mea(E)$. From 
Proposition \ref{minimax}, every lower bound of $x, y$ is
under a maximal lower bound of $x, y$. Let $u, v$ be 
maximal lower bounds of $x, y $.
By Proposition \ref{ocmdcduya}, $[ 0, x]$ and $[ 0, y]$ are 
complete MV-effect algebras. Denote by 
$z=u\wedge_{[0,x]} v\in [0, y]$. 
Then $z\leq u\wedge_{[0,y]} v$. By a symmetric argument 
we get that $u\wedge_{[0,y]} v\leq z$, i.e. 
$z= u\wedge_{[0,y]} v$. We 
may assume that $z=0$ otherwise we could shift $x, y, u, v, z$ by $z$. Since $u\comp_{[0, x]} v$ and $u\comp_{[0, y]} v$ 
we get that $u\oplus v\leq x$ exists and 
$u\oplus v=u\vee_{[0,x]} v=u\vee_{[0,y]} v\leq x, y$, 
$u\leq u\oplus v$, $v\leq u\oplus v$. Therefore $u=v=0$, 
i.e., any two maximal lower bounds of $x, y$ coincide. 
\end{proof}

In what follows we will extend and modify \cite[Lemma 20, Proposition 21]{jenca} for
an orthocomplete homogeneous  effect algebra $E$.

\begin{proposition}\label{meetmodjen} 
Let $E$ be a  homogeneous meager-orthocomplete sharply dominating 
 effect algebra and let $x, y\in E$ 
are in the same block $B$ of $E$ such that $x\wedge_{B} y=0$. Then 
$\widehat{x}\wedge \widehat{y}=0$. 
\end{proposition}
\begin{proof} Note that $\widehat{x}, \widehat{x}\ominus x, 
\widehat{y}, \widehat{y}\ominus y\in B$.
Let us first check that $\widehat{x}\wedge_B {y}=0$. 
By Proposition \ref{corcduya} applied to the 
element $\widehat{x}\ominus x$ we have from Lemma 
\ref{xssuplem} that 
there is  an orthogonal family $(x_j)_{j\in J}$ such that, for 
all $j\in J$, 
$x_j\leq x, x_j\in B\cap [0, \widehat{x}\ominus x]$ and 
$\bigoplus_{j\in J}^{\Mea(E)}  x_j= %
\widehat{x}\ominus x\in \Mea(E)$. 
Let $w\in B$, 
$w\leq \widehat{x}$ and $w\leq y$. Since $w\leq x\oplus (\widehat{x}\ominus x)$ 
we can find by the Riesz decomposition property of $B$ 
elements $w_1, w_2\in B$ such that $w=w_1\oplus w_2$, $w_1\leq x$ and $w_2\leq \widehat{x}\ominus x$. 
Therefore $w_1\leq x\wedge y=0$ implies $w=w_2\leq \widehat{x}\ominus x$.
Hence by Corollary \ref{secorssxhcmea} we obtain that 
there exists an orthogonal family $(u_j)_{j\in J}$ such that 
$\bigoplus_{j\in J}^{\Mea(E)}  u_j=w$ and $u_j\leq x_j$ for all $j\in J$. This yields that 
$u_j=0$ for all $j\in J$, i.e., $w=0$. It follows that $\widehat{x}\wedge_B {y}=0$. 
Applying the above considerations to $\widehat{x}\wedge_B {y}=0$ we get that 
$\widehat{x}\wedge_B \widehat{y}=0$. 

Now, since $\widehat{x}\comp_{B}\widehat{y}$ and $\widehat{x}\wedge_B \widehat{y}$ 
exists we have that $\widehat{x}\oplus (\widehat{y} \ominus %
(\widehat{x}\wedge_B \widehat{y}))= \widehat{x}\oplus \widehat{y}$ exists. 
It follows that $\widehat{x}\leq (\widehat{y})'$, i.e., 
$\widehat{x}\wedge \widehat{y}%
\leq (\widehat{y})'\wedge \widehat{y}=0$.  
\end{proof}

\begin{theorem}\label{blocksar}
Let $E$ be a  homogeneous 
meager-orthocomplete sharply dominating effect algebra.
Then every block $B$ in $E$ is a lattice.
\end{theorem}

\begin{proof}
Let $B$ be a block and $y,z \in B$. Then by Corollary \ref{duscduya} 
$y=y_S\oplus y_M$, $z=z_S\oplus z_M$, 
$y_S, z_S\in B\cap \Sh(E)$, $y_M, z_M\in B\cap \Mea(E)$. 
Let us put $c=(y_S\wedge_{B}z_S)\oplus %
({y}_S\wedge_B {z}_M)\oplus ({z}_S\wedge_B {y}_M)%
\oplus (y_M\wedge_B z_M)\in B$. Note that all the summands of 
$c$ exist in virtue of Statement \ref{gejzapulm}, (i) and 
Propositions \ref{corcduya}, \ref{dusminimax}. Then clearly 
$c$ is well defined since by Statement \ref{gejzapulm}, (i)
$(y_S\wedge_{B}z_S)\oplus ({z}_S\wedge_B {y}_M)=%
z_{S}\wedge_B (y_S\oplus y_M)\leq z_S$ and 
 by Statement \ref{gejzapulm}, (ii) 
 $z_M=z_M\wedge_B (y_S\oplus {y_S}')=%
(z_M\wedge_B y_S)\oplus (z_M\wedge_B {y_S}')\geq %
(z_M\wedge_B y_S)\oplus (z_M\wedge_B y_M)$. Therefore 
also $c\leq z$ and by a symmetric argument we get that 
$c\leq y$. Hence $c\leq y, z$. 

Let us show that $c=y\wedge_B z$. Assume now that $v\leq  y, z$, $v\in B$. 

By the Riesz decomposition property of $B$ there are elements 
$v_1, v_2\in B$ such that $v_1\leq {y}_S$, $v_2\leq {y}_M$ and 
$v=v_1\oplus v_2\leq z$. Again by the Riesz decomposition property we can find elements $v_{11}, v_{12}, v_{21}, v_{22}\in B$ such that 
$v_{11}\leq z_{S}$, $v_{12}\leq z_{M}$, 
$v_{21}\leq z_{S}$, $v_{22}\leq z_{M}$ and 
$v_1=v_{11}\oplus v_{12}$, 
$v_2=v_{21}\oplus v_{22}$. Hence 
$v=v_{11}\oplus v_{12} \oplus v_{21}\oplus v_{22}\leq (y_S\wedge_{B}z_S)\oplus %
({y}_S\wedge_B {z}_M)\oplus ({z}_S\wedge_B {y}_M)%
\oplus (y_M\wedge_B z_M)$.

Consequently, $c$ is the infimum of $y,z$. This yields that $B$ is a lattice. 
\end{proof}

\begin{corollary}\label{ohbmv}
Every block in a homogeneous  orthocomplete effect algebra
is an MV-algebra.
\end{corollary}

\begin{theorem}\label{archimde}
Let $E$ be a  homogeneous 
meager-orthocomplete sharply dominating 
 effect algebra.
Then $E$ is Archimedean.
\end{theorem}
\begin{proof}
In virtue of Proposition~\ref{archim}, it is sufficient to check that
\(\Mea(E)\) is Archimedean.
Suppose \(\operatorname{ord}(y)=\infty\).
By Corollary~\ref{soucethat}, \(ky\leq \wwidehat{y}\) for all \(k\in \mathbb
N\), and therefore \((k-1)y=ky\ominus y\leq \wwidehat{y}\ominus y\in
\Mea(E)\) for all \(k\in \mathbb N \subseteq\{0\}\).
Hence there exists \(\bigvee \{ky\mid y\in \mathbb N\}\) in \(\Mea(E)\).
By (IDL), y=0.
\end{proof}

Recall that a {\em Heyting} algebra (see \cite{foulishea}) is a system 
$(L, \leq , 0, 1, \wedge , \vee , \Rightarrow)$ consisting
of a bounded lattice $(L, \leq , 0, 1, \wedge , \vee)$ and 
a binary operation $\Rightarrow : L \times L \to L$,
called the {\em Heyting implication connective}, such that 
$x \wedge y \leq z$ iff $x \leq (y \Rightarrow  z)$
for all $x, y, z \in L$.  The {\em Heyting negation mapping} 
$^{*}: L \to L$ is defined by $x^{*}= (x \Rightarrow  0)$
for all $x \in L$. The set $L^{*}= \{ x^{*}\mid x \in L \}$ 
is called the {\em Heyting center of $L$}. 

 A {\em pseudocomplementation} on a bounded lattice 
$(L, \leq , 0, 1, \wedge , \vee)$ is a mapping 
$^{*}: L \to L$ such that, for all $x, y \in L$, 
$x \wedge y = 0$ iff  $x \leq y^{*}$. Then, for a {Heyting} algebra $L$, 
the Heyting negation is a pseudocomplementation on L. 

\begin{definition}\label{heafoul}{\rm \cite{foulishea}
A {\em Heyting effect algebra}  is a lattice effect
algebra $E$ that, as a bounded lattice, is also a Heyting algebra such that the
Heyting center $E^{*}$ coincides with the  center $\C(E)$ of the effect algebra $E$.}
\end{definition}

{\renewcommand{\labelenumi}{{\normalfont  (\roman{enumi})}}
\begin{statement} {\rm \cite[Theorem 5.2]{foulishea}} 
Let $E$ be a lattice effect algebra. Then the following
conditions are  equivalent:
\begin{enumerate}
\settowidth{\leftmargin}{(iiiii)}
\settowidth{\labelwidth}{(iiiii)}
\settowidth{\itemindent}{(iii)}
\item $E$ is a Heyting effect algebra.
\item $E$ is an MV-effect algebra with a pseudocomplementation $^{*}: E \to E$.
\end{enumerate}
\end{statement}}

\begin{theorem}\label{ycoveredhea}
Let $E$ be a  homogeneous 
meager-orthocomplete sharply dominating 
 effect algebra. 
Then  every block in $E$  is an Archimedean  Heyting effect algebra.
\end{theorem}
\begin{proof} Let $B$ be a block of $E$. Then by Theorems \ref{blocksar} and 
\ref{archimde} we have that $B$ is an Archimedean  MV-effect algebra. 
Let us define a pseudocomplementation $^{*}: B \to B$ on $B$. For any 
$x\in B$ we put $x^{*}=\widehat{x}'\in \C(B)$. Assume that $x, y\in B$. Then by 
Proposition \ref{meetmodjen}  
$x\wedge_B y=0$ iff $\widehat{x}\wedge_B \widehat{y}=0$ iff 
$\widehat{x}\leq \widehat{y}'$ iff $\widehat{x}\leq {y}^{*}\in \Sh(E)$ iff 
${x}\leq {y}^{*}$. Therefore $B$ is an Archimedean   Heyting  effect algebra.
\end{proof}

\begin{corollary}\label{coveredhea}
Let $E$ be a  homogeneous 
meager-orthocomplete sharply dominating 
 effect algebra. 
Then  $E$  can be covered  by Archimedean  Heyting effect algebras.
\end{corollary}
\begin{proof} Every homogeneous effect algebra is covered by its blocks.
\end{proof}

{\renewcommand{\labelenumi}{{\normalfont  (\roman{enumi})}}
 \begin{proposition}\label{modjen} Let $E$ be a homogeneous 
meager-orthocomplete sharply dominating   effect algebra and let $x, y\in\Mea(E)$ 
and $v\in E$ such that $x, y$ and $v$ are in the same block $B$ of $E$. Then 
\begin{enumerate}
\settowidth{\leftmargin}{(iiiii)}
\settowidth{\labelwidth}{(iii)}
\settowidth{\itemindent}{(ii)}
\item We have that
$v\wedge y$ exists and $v\wedge_{B} y=v\wedge y$.
\item If $\wwidehat{x} = \wwidehat{y}$ then $\wwidehat{x}\ominus (x\wedge y)\in \Mea(E)$.
\item If $\wwidehat{x} = \wwidehat{y}$ then $x\vee_{\Mea(E)} y$ exists and 
$x\vee_{\Mea(E)} y=x\vee_{B} y=x\vee_{[0, \wwidehat{x}]}y$.
\item If $\wwidehat{x} = \wwidehat{y}$ then $\wwidehat{x\wedge y}=\wwidehat{x}$.
\item If $x\leq v$ and $v=v_S\oplus v_M$ such that $v_S\in \Sh(E)$ and 
$v_M\in \Mea(E)$ then $x=(x\wedge v_S)\oplus (x\wedge v_M)$.
\end{enumerate}
\end{proposition}}
\begin{proof} According to Proposition \ref{corcduya}, $\wwidehat{x},\wwidehat{y}\in B$.
By Proposition \ref{dusminimax}, $x\wedge y=x\wedge_{\Mea(E)} y$ exists and belongs to
$\Mea(E)$. In virtue of Corollary \ref{duscduya}, it belongs to $B$.

\noindent
(i): Let $u\leq v$ and $u\leq y$. Since $u\in [0, y]\subseteq B$ we have that 
$u\leq v\wedge_{B} y$.

\noindent{}(ii): Since $x, y \in \Mea(E)$ we have 
that $x\wedge y$ exists, $x\leq \wwidehat{x}$, $y\leq \wwidehat{x}$ and hence also 
$\wwidehat{x}\ominus (x\wedge y)\leq \wwidehat{x}$ exists. Hence  
$(\wwidehat{x}\ominus (x\wedge y))\oplus (x\wedge y)=\wwidehat{x}$.
$B$ contains
$0, x, y, x\wedge y, \wwidehat{x}\ominus (x\wedge y), \wwidehat{x}$. 
Let us put $z=\wwidehat{x}\ominus (x\wedge y)$. Then 
$z=z_S\oplus z_M$, $z_S\in \Sh(E)\cap B=\C(B), z_S\leq \wwidehat{x}$ 
and $z_M\in \Mea(E)$. 
Since $z_S$ is central in $B$ we have that 
$(\wwidehat{x}\ominus (x\wedge y))\wedge_B z_S)\oplus ((x\wedge y)\wedge_B z_S)=z_S$. 
From (i) we get that 
$z_S\oplus (x\wedge (y\wedge z_S))=z_S$ and 
$z_S\oplus (y\wedge (x\wedge z_S))=z_S$.  Hence 
$0=x\wedge (y\wedge z_S)=y\wedge (x\wedge z_S)$.  It follows from (i) that 
$0=\wwidehat{x}\wedge (y\wedge z_S)=y\wedge z_S%
=\wwidehat{y}\wedge (x\wedge z_S)=x\wedge z_S$. 
We have that 
$z_S=z_S\wedge \wwidehat{x}= %
z_S\wedge_B (x\oplus (\wwidehat{x}\ominus x))=%
(z_S\wedge_B x)\oplus (z_S\wedge_B (\wwidehat{x}\ominus x))=%
z_S\wedge_B (\wwidehat{x}\ominus x)\leq \wwidehat{x}\ominus x$. 
Since $\wwidehat{x}\ominus x\in \Mea(E)$ we obtain that $z_S=0$ and 
$z_M=\wwidehat{x}\ominus (x\wedge y)\in \Mea(E)$.

\noindent{}(iii): Since $\wwidehat{x}\ominus x, \wwidehat{x}\ominus y \in \Mea(E)$ we have 
by (ii) that  $\wwidehat{x}\ominus \left((\wwidehat{x}\ominus x)\wedge (\wwidehat{x}\ominus y)\right)=%
x\vee_{[0, \wwidehat{x}]}y\in \Mea(E)$ exists. Let $z\geq x, y$, $z\in \Mea(E)$. Then 
$u=z\wedge (x\vee_{[0, \wwidehat{x}]}y)\in \Mea(E)$ exists, $u\geq x, y$, 
$u\leq x\vee_{[0, \wwidehat{x}]}y\leq \wwidehat{x}$. Hence $z\geq u\geq x\vee_{[0, \wwidehat{x}]}y$.

\noindent{}(iv): Clearly, 
${x\wedge y}\leq \wwidehat{x\wedge y}\leq \wwidehat{x}$. This yields that 
$\wwidehat{x}\ominus \wwidehat{x\wedge y}\leq \wwidehat{x}\ominus {x\wedge y}\in \Mea(E)$, 
$\wwidehat{x}\ominus \wwidehat{x\wedge y}\in \Sh(E)$. 
It follows that $\wwidehat{x}\ominus \wwidehat{x\wedge y}=0$, i.e., 
$\wwidehat{x\wedge y}=\wwidehat{x}$. 

\noindent{}(v):
Then $v_S\in \C(B)$ and $v_M\in B$. Hence 
$x=(x\wedge_B v_S)\oplus (x\wedge_B (v_S)')$, 
$x\wedge_B v_S, x\wedge_B (v_S)'\in B$. Moreover by (i) we have that 
$x\wedge_B v_S=x\wedge v_S$ and $x\wedge_B (v_S)'=x\wedge (v_S)'$.
Evidently, 
$x\wedge (v_S)'\leq v=v_S\oplus v_M$. Since $B$ has the Riesz 
decomposition property we have that 
$x\wedge (v_S)'=u_1\oplus u_2$, $u_1\leq v_S$ and $u_2\leq v_M$. 
But $u_1\leq x\wedge (v_S)'$ yields that $u_1=0$. Hence 
$x\wedge (v_S)'\leq v_M$. This yields that 
$x\wedge (v_S)'=x\wedge (v_S)'\wedge v_M=x\wedge  v_M$. It follows that 
$x=(x\wedge v_S)\oplus (x\wedge v_M)$.
\end{proof}

The following theorem reminds us \cite[Theorem 37]{chovanec} which 
was formulated for D-lattices.

{\renewcommand{\labelenumi}{{\normalfont  (\roman{enumi})}}
\begin{theorem} \label{tmodjensup}\label{modchov} Let $E$ be a 
homogeneous meager-orthocomplete sharply dominating  
 effect algebra and let $x, y\in\Mea(E)$. 
Then the following conditions are equivalent:
\begin{enumerate}
\settowidth{\leftmargin}{(iiiii)}
\settowidth{\labelwidth}{(iii)}
\settowidth{\itemindent}{(ii)}
\item  $x\comp y$. 
\item  $x\comp_{\Mea(E)} y$. 
\item  $x\vee_{\Mea(E)} y$ exists and 
$(x\vee_{\Mea(E)}  y)\ominus y=x \ominus (x\wedge y)$.
\end{enumerate}
\end{theorem}}
\begin{proof} (i)$\implik$(ii): Since $x\comp y$ there are $p, q, r\in E$ such that 
$x=p\oplus q$, $y=q\oplus r$ and $p\oplus q\oplus r$ exists. Clearly, 
$p, q\leq x$ and $q, r\leq y$. Hence $q\leq x\wedge y\in \Mea(E)$. 
Moreover, $x=(x\ominus (x\wedge y))\oplus (x\wedge y)$, 
$y=(y\ominus (x\wedge y))\oplus (x\wedge y)$ and 
$(x\ominus (x\wedge y))\oplus (x\wedge y)\oplus (y\ominus (x\wedge y))$ exists 
since $p\oplus q\oplus r=x\oplus r$ exists and $y\ominus (x\wedge y)\leq r=y\ominus q$. 
Let us put $z=(x\ominus (x\wedge y))\oplus (x\wedge y)\oplus (y\ominus (x\wedge y))$. 
Since $E$ is sharply dominating we have that $z=z_S\oplus z_M$, 
$z_S\in \Sh(E)$ and $z_M\in\Mea(E)$. 

Since $x\comp y$ there is a block $B$ of $E$ such that $
x, y, x\wedge y, z, z_S, z_M\in B$. Hence $z_S\in \C(B)$ and 
therefore  
$z_S= z_S\wedge \left((x\ominus (x\wedge y))\oplus (x\wedge y)\oplus (y\ominus (x\wedge y))\right)=%
((z_S\wedge x)\ominus ((z_S\wedge x)\wedge (z_S\wedge y)))%
\oplus ((z_S\wedge x)\wedge (z_S\wedge y))\oplus ((z_S\wedge y)\ominus ((z_S\wedge x)%
\wedge (z_S\wedge y)))$. Let us put $u=z_S\wedge x$ and $v=z_S\wedge y$. Then 
$z_S= (u\ominus (u\wedge v))\oplus (u\wedge v)\oplus (v\ominus (u\wedge v))$, 
$u, v\in \Mea(E)\cap B$, $\wwidehat{u}, \wwidehat{v}\in B$.

Recall first that $\wwidehat{u}=\wwidehat{v}=\wwidehat{v\ominus (u\wedge v)}=z_S$. 
This follows immediately from Statement \ref{jmpy2}, (ii).

Since $(u\ominus (u\wedge v))\wedge (v\ominus (u\wedge v))=0$ we have by 
\cite[Proposition 3.4]{pulmblok} that $z_S$ is the minimal upper bound of $u$ and $v$. 
Therefore $z_S=u\vee_{[0, z_S]} v$. From Proposition \ref{modjen}, (iii) we have that 
$u\vee_{[0, z_S]} v\in \Mea(E)$. Hence $z_S=0$ and 
$z=z_M=(x\ominus (x\wedge y))\oplus (x\wedge y)\oplus %
(y\ominus (x\wedge y))\in\Mea(E)$, i.e., $x\comp_{\Mea(E)} y$. 

\noindent{}(ii)$\implik$(iii): Assume that $x\comp_{\Mea(E)} y$, i.e., 
$z=(x\ominus (x\wedge y))\oplus_{\Mea(E)} (x\wedge y)\oplus %
(y\ominus (x\wedge y))$ exists. Again by \cite[Proposition 3.4]{pulmblok} 
we have that $z\in\Mea(E)$ is the minimal upper bound 
of $x$ and $y$. Let $m\in\Mea(E)$ be an upper bound of $x$ and $y$. 
Then $m\wedge z$ is an upper bound of $x$ and $y$, $m\wedge z\leq z$ and 
hence $m\geq m\wedge z= z$. It follows that  $(x\vee_{\Mea(E)}  y)\ominus y=%
x \ominus (x\wedge y)$.

\noindent{}{(iii)$\implik$(i)}:  It is enough to put $d=x\wedge y$ and 
$c=x\vee_{\Mea(E)}  y$. Then 
$d\leq x\leq c$ and $d\leq y\leq c$ such that 
$c\ominus x=y\ominus d$, i.e., $x\comp y$.
\end{proof}

\section{Triple Representation Theorem for orthocomplete \\
homogeneous effect algebras}\label{tretikapitola}

In what follows $E$ will be always a homogeneous meager-orthocomplete 
sharply dominating    
effect algebra. Then $\Sh(E)$ is a sub-effect algebra 
of $E$ and $\Mea(E)$ equipped with a partial 
operation $\oplus_{\Mea(E)}$ which is defined, for all $x, y\in \Mea(E)$, 
by $x\oplus_{\Mea(E)} y$ exists
if and only if $x\oplus_E y$  exists and $x\oplus_E y\in \Mea(E)$ in which 
case $x\oplus_{\Mea(E)} y=x\oplus_E y$ is a generalized  effect algebra. 
Moreover, we have a map 
$h:\Sh(E)\to 2^{\Mea(E)}$ that is given by 
$h(s)=\{x\in \Mea(E) \mid x\leq s\}$. As in \cite{jenca} for complete 
lattice effect algebras we will 
prove the following theorem. 

\medskip

\noindent{\bfseries Triple Representation Theorem}
 {\em The triple $((\Sh(E),\oplus_{\Sh(E)}),$ $(\Mea(E),\oplus_{\Mea(E)}),$ $h)$
characterizes $E$ up to isomorphism within the class of all 
homogeneous meager-orthocomplete sharply
dominating effect algebras.}
\medskip

We have to  construct an isomorphic copy of the original effect algebra
$E$ from the triple $(\Sh(E), \Mea(E), h)$. To do this we will first
construct  the
following mappings in terms of the triple.

\medskip

\begin{tabular}{@{}c@{}l}
&\begin{minipage}{0.9454\textwidth}
\begin{itemize}
\item[(M1)] The mapping\quad $\wwidehat{\phantom{x}}:\Mea(E) \to \Sh(E)$.
\item[(M2)] For every $s\in \Sh(E)$, a partial mapping 
$\pi_{s}:\Mea(E) \to h(s)$, which is given by 
$\pi_{s}(x)=x\wedge_{E} s$ whenever $\pi_{s}(x)$ is defined.
\item[(M3)] The mapping\, $R:\Mea(E) \to \Mea(E)$ 
given by $R(x)=\wwidehat{{x}}\ominus_E x$.
\item[(M4)] The partial mapping\, \mbox{$S:\Mea(E) \times \Mea(E)\to \Sh(E)$} 
given by $S(x, y)$ is defined if and only if 
the set ${\mathscr S}(x, y)=\{z\in \Sh(E)\mid z\wedge x \ \text{and}\ z\wedge y\ \text{exist}, %
z=(z\wedge x)\oplus_E (z\wedge y)\}$ has a top element 
$z_0\in {\mathscr S}(x, y)$ 
in which case $S(x, y)=z_0$.
\end{itemize}
\end{minipage}
\end{tabular}

\medskip

Since $E$ is sharply dominating 
we have that, for all $x\in \Mea(E)$, 
$$\wwidehat{x}=\bigwedge_{E}\{s\in \Sh(E) \mid x\in h(s)\}=%
\bigwedge_{\Sh(E)}\{s\in \Sh(E) \mid x\in h(s)\}.$$ 

{\renewcommand{\labelenumi}{{\normalfont  (\roman{enumi})}}
\begin{lemma}\label{m2}
Let $E$ be a homogeneous meager-orthocomplete sharply
dominating  effect algebra, $s\in \Sh(E)$ and  $x\in \Mea(E)$. 
Then 
\begin{enumerate}
\settowidth{\leftmargin}{(iiiii)}
\settowidth{\labelwidth}{(iii)}
\settowidth{\itemindent}{(ii)}
\item $\bigvee_{M(E)}\{y\in \Mea(E) \mid y\leq x, y\leq s\}$ exists.
Moreover, if $x\wedge_{E} s$ exists then $x\wedge_{E} s\in \Mea(E)$ and 
$$x\wedge_{E} s=\bigvee_{E}\{y\in E \mid y\leq x, y\leq s\}=%
\bigvee_{M(E)}\{y\in \Mea(E) \mid y\leq x, y\leq s\}.$$
\item If $x\comp s$ then $x\wedge_{E} s$ exists and 
$$x\wedge_{E} s=\bigvee_{M(E)}\{y\in \Mea(E) \mid y\leq x, y\leq s\}.$$
\end{enumerate}
\end{lemma}
\begin{proof} (i): Note that 
$\{y\in \Mea(E) \mid y\leq x, y\leq s\}\subseteq [0, x]$. Since 
from Proposition \ref{ocmdcduya} we have that   $[0, x]$ is 
a complete MV-effect algebra we get that 
$z=\bigvee_{[0, x]}\{y\in \Mea(E) \mid y\leq x, y\leq s\}$ exists. 
Let us show that $z=\bigvee_{\Mea(E)}\{y\in \Mea(E) \mid y\leq x, y\leq s\}$. 
Let $m\in \Mea(E)$ such that $m$ is an upper bound of the set 
$\{y\in \Mea(E) \mid y\leq x, y\leq s\}$. Since $m\wedge_E z\in \Mea(E)$ 
exists and $m\wedge_E z\in [0, x]$ is an upper bound of the set 
$\{y\in \Mea(E) \mid y\leq x, y\leq s\}$ we have that 
$z\leq m\wedge_E z \leq m$. 


Now, assume that $x\wedge_{E} s$ exists  and 
let us check that $x\wedge_{E} s=z$. Clearly, 
$x\wedge_{E} s=\bigvee_{E}\{y\in E \mid y\leq x, y\leq s\}\in  \Mea(E)$, 
i.e., $y\leq x\wedge_{E} s \leq z$ for all $y\in \Mea(E)$ such that 
$y\leq x$ and  $y\leq s$. This yields $x\wedge_{E} s=z$.

\noindent{}(ii): By Statement \ref{gejzasum}, (e) there is some block 
$B$ of $E$ such that $x, s\in B$. Hence $s\in\C(B)$ by 
Statement \ref{gejzasum}, (g). This yields that $x\wedge_{B} s\in B$ exists 
and since $x\wedge_{B} s\leq x$ we have that $x\wedge_{B} s\in \Mea(E)$. 
>From Corollary \ref{duscduya} we know that 
$[0, x]_{E}\subseteq B$. Hence 
$x\wedge_{B} s\in \{y\in \Mea(E) \mid y\leq x, y\leq s\}\subseteq B$ 
and $z\in B$. This invokes that $x\wedge_{B} s \leq z$.
Then $z\wedge_{B} s\in B$ exists in $B$, 
$z\wedge_{B} s\leq x$, $z\wedge_{B} s\leq s$  and, for all 
 $y\in \Mea(E)$ such that $y\leq x$ and  $y\leq s$, we have that 
 $y\leq z\wedge_{B} s$. Hence $z\leq z\wedge_{B} s\leq s$. 
 Altogether $z=x\wedge_{B} s$. Let us check that $z=x\wedge_{E} s$. 
 Assume that $g\in E$, $g\leq s$ and $g\leq x$. Then $g\in B$ and 
 therefore $g\leq x\wedge_{B} s$. It follows that
 $z=x\wedge_{E} s=x\wedge_{B} s$.
\end{proof}

Hence, for all $s\in \Sh(E)$ and for all $x\in \Mea(E)$, we 
put $z=\bigvee_{M(E)}\{y\in \Mea(E) \mid y\leq x, y\in h(s)\}$. 
Then $\pi_{s}(x)$ is defined if $z\in h(s)$ in which case 
$$
\begin{array}{@{}r@{}l}
\pi_{s}&(x)=z = x\wedge_{E} s=%
\bigvee_{E}\{y\in E \mid y\leq x, y\leq s\}\\
\phantom{\text{\huge I}}&=\bigvee_{E}\{y\in \Mea(E) \mid y\leq x, y\in h(s)\}
=\bigvee_{M(E)}\{y\in \Mea(E) \mid y\leq x, y\in h(s)\}.
\end{array}
$$

Now, let us construct the mapping  $R$ as in \cite{jenca}. 

{\renewcommand{\labelenumi}{{\normalfont  (\roman{enumi})}}
\begin{lemma}\label{m3}
Let $E$ be a  homogeneous meager-orthocomplete sharply
dominating  effect algebra   and let $x\in \Mea(E)$. 
Then $y=\wwidehat{x}\ominus x$ is the only element such that
\begin{enumerate}
\settowidth{\leftmargin}{(iiiii)}
\settowidth{\labelwidth}{(iii)}
\settowidth{\itemindent}{(ii)}
\item $y\in \Mea(E)$ such that $\wwidehat{y}=\wwidehat{x}$. 
\item $x \oplus_{\Mea(E)} (y\ominus_{\Mea(E)} (x\wedge y))$ 
exists and 
$x \oplus_{\Mea(E)} (y\ominus_{\Mea(E)} (x\wedge y))\in %
h(\wwidehat{x})$.
\item For all $z \in h(\wwidehat{x})$, 
$z \oplus_{\Mea(E)} x \in h(\wwidehat{x})$
if and only if $z \leq y$ and 
$\wwidehat{y\ominus_{\Mea(E)} z}=\wwidehat{x}$.
\end{enumerate}
\end{lemma}}
\begin{proof}  Let us prove that $y=\wwidehat{x}\ominus x$ satisfies 
\mbox{(i)-(iii)}. 
By Statement \ref{jmpy2} we get that (i) is satisfied. Evidently, 
$x\wedge y$ exists and $x\oplus y=\wwidehat{x}$. 
Invoking Theorem \ref{modchov} we obtain that 
$x\oplus_{\Mea(E)} (y\ominus_{\Mea(E)} (x\wedge y)) \in \Mea(E)$. 
Let $z \in h(\wwidehat{x})$ and assume that 
$z \oplus_{\Mea(E)} x \in h(\wwidehat{x})$. Since 
$z \oplus_{\Mea(E)} x=z \oplus x \leq x\oplus y=\wwidehat{x}$ we get 
that $z\leq y$. Moreover,  $(x\oplus z)\oplus (y\ominus z)=\wwidehat{x}$ 
yields again by Statement \ref{jmpy2} that $\wwidehat{y\ominus z}=\wwidehat{x}$. 
Now, let $z \in h(\wwidehat{x})$, $z \leq y$ and 
$\wwidehat{y\ominus_{\Mea(E)} z}=\wwidehat{x}$. Then 
$\wwidehat{x}=x\oplus y= x\oplus (y\ominus z)\oplus z= (x\oplus z)\oplus (y\ominus z)$. 
Since $y\ominus z\in\Mea(E)$ we have that $x\oplus z\in \Mea(E)$. Hence 
$z \oplus_{\Mea(E)} x \in h(\wwidehat{x})$. 

Let us verify that $y=\wwidehat{x}\ominus x$ is the only element satisfying 
\mbox{(i)-(iii)}. Let the elements $y_1$ and $y_2$ of $E$ satisfy 
\mbox{(i)-(iii)} and $y_1\comp y_2$. Let us put $u=y_1\wedge y_2$. 
Then $\wwidehat{u}=\wwidehat{x}$ by Proposition \ref{modjen}, (iv). By (ii) 
for $y_1$ we know that $x\vee_{\Mea(E)} y_1\in \Mea(E)$ exists. Since 
$[0, x\vee_{\Mea(E)} y_1]$ is a complete lattice we have that 
$x\vee_{[0, x\vee_{\Mea(E)} y_1]} (y_1\wedge y_2)\in \Mea(E)$ exists. 
Hence also 
$x\vee_{\Mea(E)} (y_1\wedge y_2)\in \Mea(E)$ exists and 
$x\vee_{[0, x\vee_{\Mea(E)} y_1]} (y_1\wedge y_2)=%
x\vee_{\Mea(E)} (y_1\wedge y_2)$. This yields that $x \comp (y_1\wedge y_2)$ 
and from Theorem \ref{modchov} we get that 
$x\oplus_{\Mea(E)} ((y_1\wedge y_2)\ominus_{\Mea(E)} (x\wedge (y_1\wedge y_2)))$ 
exists and $x\oplus_{\Mea(E)} ((y_1\wedge y_2)\ominus_{\Mea(E)} %
(x\wedge (y_1\wedge y_2)))=x\vee_{\Mea(E)} (y_1\wedge y_2)$. 
Clearly, for any  $z \in h(\wwidehat{x})$, we have 
$z \oplus_{\Mea(E)} x \in h(\wwidehat{x})$ iff 
$\left(z \leq y_1 \ \text{and} \ 
\wwidehat{y_1\ominus_{\Mea(E)} z}=\wwidehat{x}\right)$ and 
$\left(z \leq y_2\ \text{and }\ 
\wwidehat{y_2\ominus_{\Mea(E)} z}=\wwidehat{x}\right)$ iff (by 
Proposition \ref{modjen}, (iv)) %
$\left(z \leq (y_1\wedge y_2)\ \text{and}\ 
\wwidehat{(y_1\wedge y_2)\ominus_{\Mea(E)} z}=\wwidehat{x}\right)$. 
Hence $u=y_1\wedge y_2$ satisfies \mbox{(i)-(iii)}. Let us put 
$t=y_1\ominus u$. 

We will prove that, for all $n\in {\mathbb N}$, 
$(nt) \oplus_{\Mea(E)} x \in h(\wwidehat{x})$. For $n=0$ the 
statement is true. Assume that the statement is valid for some 
$n\in {\mathbb N}$. Then by (iii) for $u$ we have that 
$nt \leq u$ and 
$\wwidehat{u\ominus_{\Mea(E)} (nt)}=\wwidehat{x}$. Since 
$u\oplus t=y_1\in \Mea(E)$ we get that $(n+1)t$ is defined, 
$(n+1)t \leq y_1\leq \wwidehat{x}$, 
$u\ominus_{\Mea(E)} (nt)=y_1\ominus_{\Mea(E)} ((n+1)t)$  and 
hence $\wwidehat{y_1\ominus_{\Mea(E)} ((n+1)t)}=\wwidehat{x}$. 
By (iii) for $y_1$ we obtain that 
$((n+1)t) \oplus_{\Mea(E)} x \in h(\wwidehat{x})$. In particular, 
$nt$ exists  for all $n\in {\mathbb N}$. Since $E$ is 
Archimedean, i.e., $t=0$ and $y_1\wedge y_2=y_1$. This yields 
that $y_1\leq y_2$. Interchanging $y_1$ with $y_2$ we get that 
$y_2\leq y_1$, i.e., $y_1=y_2$.

Now, let us assume that some $y$ satisfies \mbox{(i)-(iii)} and 
put $y_1=y$, $y_2=\wwidehat{x}\ominus x$. Since 
$x \oplus_{\Mea(E)} (y\ominus_{\Mea(E)} (x\wedge y))\leq \wwidehat{x}$ 
exists by (ii) we have that $x\comp y$ and this yields that $\wwidehat{x}\comp y$. 

By \cite[Theorem 36]{chovanec} we get that 
$y_1=y\comp (\wwidehat{x}\ominus x)=y_2$. Therefore $y=\wwidehat{x}\ominus x$.
\end{proof}

What remains is the partial mapping $S$. Let $x, y\in \Mea(E)$. 
Note that by Statement \ref{jmpy2}, (ii) and Lemmas \ref{m2} and \ref{m3} 
${\mathscr S}(x, y)=\{z\in \Sh(E)\mid %
z=(z\wedge x)\oplus_E (z\wedge y)\}=\{z\in \Sh(E)\mid %
\pi_z(x) \ \text{and}\ \pi_z(x)\ \text{are defined},  
z=\wwidehat{\pi_z(x)}\ \text{and}\ R(\pi_z(x))=\pi_z(y)\}$.
Hence whether $S(x, y)$ is defined or not we are able to decide 
in terms of the triple. Since the eventual top element $z_0$ of 
${\mathscr S}(x, y)$ is in $\Sh(E)$ our definition of ${S}(x, y)$ 
is correct.

\begin{lemma}\label{pommeag} Let $E$ be a homogeneous meager-orthocomplete sharply
dominating  effect algebra, $x, y\in \Mea(E)$. Then $x\oplus_{E} y$ exists 
in $E$
iff $S(x, y)$ is defined in terms of the triple $(\Sh(E), \Mea(E), h)$ and 
$(x\ominus_{\Mea(E)} (S(x, y)\wedge x))%
\oplus_{\Mea(E)} (y\ominus_{\Mea(E)} (S(x, y)\wedge y))$ 
exists in $\Mea(E)$ such that %
$(x\ominus_{\Mea(E)} (S(x, y)\wedge x))%
\oplus_{\Mea(E)} (y\ominus_{\Mea(E)} (S(x, y)\wedge y))%
\in h(S(x, y)')$. Moreover, in 
that case
$$x\oplus_{E} y=\underbrace{S(x, y)}_{\in \Sh(E)}\oplus_{E} %
(\underbrace{(x\ominus_{\Mea(E)} (S(x, y)\wedge x))%
\oplus_{\Mea(E)} (y\ominus_{\Mea(E)} (S(x, y)\wedge y))}_{\in \Mea(E)}).$$
\end{lemma}
\begin{proof} Assume first that $x\oplus_{E} y$ exists 
in $E$ and let us put $z=x\oplus_{E} y$. Since $E$  is 
sharply dominating we have that
$z=z_S\oplus_{E}  z_M$ such that $z_S\in \Sh(E)$ and 
$z_M\in \Mea(E)$. Since $x\comp y$ by Statement 
\ref{gejzasum}, (e) there is  a  block $B$ 
of $E$ such that $x, y, z\in B$. 
By Statement \ref{gejzaoch}, (i) we obtain 
that $z_S,  z_M\in B$. Therefore $z_S\in \C(B)$ and by Statement 
\ref{gejzapulm}, (i) we have that 
 $z_S= z_S \wedge (x\oplus_{E} y)= z_S \wedge (x\oplus_{B} y)= 
 (z_S \wedge_{B} x)\oplus_{B} (z_S \wedge_{B} y)= %
(z_S \wedge x)\oplus_{E} (z_S \wedge y)$. Hence $z_S\in {\mathscr S}(x, y)$. 
Now, assume that $u\in {\mathscr S}(x, y)$. Then 
$u=(u \wedge x)\oplus_{E} (u \wedge y)\leq x\oplus_{E} y$. Since 
$u\in \Sh(E)$ we have that $u\leq z_S$, i.e., $z_S$ is the top 
element of ${\mathscr S}(x, y)$. Moreover, we have 
$$
\begin{array}{@{}r@{\,}c@{\,}l}
z_S\oplus_E z_M&=&x\oplus_{E} y\\
&=&\biggl(\bigl(S(x, y)\wedge x\bigr)\oplus_{E} %
\bigl(x\ominus_{E} (S(x, y)\wedge x)\bigr)\biggr)%
\oplus_E \\[4mm]
& &\biggl(\bigl((S(x, y)\wedge y)\bigr)\oplus_{E}
 \bigl(y\ominus_{E} (S(x, y)\wedge y)\bigr)\biggr)\\
\multicolumn{3}{r}{=S(x, y)\oplus_{E} %
\biggl(\bigl(x\ominus_{\Mea(E)} (S(x, y)\wedge x)\bigr)%
\oplus_{E} \bigl(y\ominus_{\Mea(E)} (S(x, y)\wedge y)\bigr)\biggr).}
 \end{array}
$$
It follows that $z_M=\bigl(x\ominus_{\Mea(E)} (S(x, y)\wedge x)\bigr)%
\oplus_{E} \bigl(y\ominus_{\Mea(E)} (S(x, y)\wedge y)\bigr)$, i.e., 
$z_M=\bigl(x\ominus_{\Mea(E)} (S(x, y)\wedge x)\bigr)%
\oplus_{\Mea(E)} \bigl(y\ominus_{\Mea(E)} (S(x, y)\wedge y)\bigr)$ and evidently 
$z_M\in h(z_S')$.

Conversely, let us assume that $S(x, y)$ is defined in terms of  $(\Sh(E), \Mea(E), h)$,  
$\bigl(x\ominus_{\Mea(E)} (S(x, y)\wedge x)\bigr)%
\oplus_{\Mea(E)} \bigl(y\ominus_{\Mea(E)} (S(x, y)\wedge y)\bigr)$ 
exists in $\Mea(E)$ and %
$\bigl(x\ominus_{\Mea(E)} (S(x, y)\wedge x)\bigr)%
\oplus_{\Mea(E)} \bigl(y\ominus_{\Mea(E)} (S(x, y)\wedge y)\bigr)%
\in h(S(x, y)')$. Then 
$\bigl(x\ominus_{\Mea(E)} (S(x, y)\wedge x)\bigr)%
\oplus_{\Mea(E)} \bigl(y\ominus_{\Mea(E)} (S(x, y)\wedge y)\bigr)%
\leq S(x, y)'$, i.e., 
$$
\begin{array}{r@{\,}c@{\,}l}
z&=&S(x, y)\oplus_{E}%
\biggl(\bigl(x\ominus_{\Mea(E)} (S(x, y)\wedge x)\bigr)%
\oplus_{\Mea(E)} \bigl(y\ominus_{\Mea(E)} %
(S(x, y)\wedge y)\bigr)\biggr)\\[4mm]
&=&\bigl((S(x, y)\wedge x)\oplus_{E} (S(x, y)\wedge y)\bigr) \oplus_{E}\\[1mm]%
&&\biggl(\bigl(x\ominus_{E} (S(x, y)\wedge x)\bigr)%
\oplus_{E} \bigl(y\ominus_{E}%
(S(x, y)\wedge y)\bigr)\biggr)=x\oplus_{E} y\\
\end{array}
$$
\noindent{}is defined.
\end{proof}

{\renewcommand{\labelenumi}{{\normalfont  (\roman{enumi})}}
\begin{theorem}\label{tripletheor}
Let E be a homogeneous meager-orthocomplete sharply
dominating  effect algebra. 
Let $\Tea(E)$ be a subset
of $\Sh(E) \times \Mea(E)$ given by
$$\Tea(E) =\{(z_S, z_M)\in \Sh(E) \times \Mea(E) \mid z_M \in  h(z_S')\}.$$
Equip $\Tea(E)$ with a partial binary operation $\oplus_{\Tea(E)}$ with 
$(x_S, x_M) \oplus_{\Tea(E)} (y_S, y_M)$ is defined if and  only if 
\begin{enumerate}
\item $S(x_M, y_M)$ is defined,
\item $z_S=x_S\oplus_{\Sh(E)} y_S \oplus_{\Sh(E)} S(x_M, y_M)$ is defined,
\item $z_M=\bigl(x_M\ominus_{\Mea(E)} (S(x_M, y_M)\wedge x_M)\bigr)%
\oplus_{\Mea(E)} \bigl(y_M\ominus_{\Mea(E)}%
(S(x_M, y_M)\wedge y_M)\bigr)$ is defined,
\item $z_M\in h(z_S')$. 
\end{enumerate}
In this case $(z_S, z_M)=(x_S, x_M) \oplus_{\Tea(E)} (y_S, y_M)$.
Let\/ $0_{\Tea(E)}=(0_E,0_E)$ and\/ $1_{\Tea(E)}=(1_E,0_E)$. Then 
$\Tea(E)=(\Tea(E), \oplus_{\Tea(E)}, 0_{\Tea(E)},$ $1_{\Tea(E)})$ 
is an effect algebra and the mapping 
$\varphi:E\to \Tea(E)$  given by 
$\varphi(x) = (\widetilde{x}, x\ominus_{E} \widetilde{x})$   
is an isomorphism of effect algebras.
\end{theorem}}
\begin{proof} Evidently, $\varphi$ is correctly defined since, 
for any $x\in E$, we have that $x=\widetilde{x}\oplus_{E} (x\ominus \widetilde{x})=%
x_S\oplus_{E} x_M$, $x_S\in \Sh(E)$ and $x_M\in \Mea(E)$. Hence 
$\varphi(x) = (x_S, x_M)\in \Sh(E) \times \Mea(E)$ and $x_M\in h(x_S')$. Let us check that $\varphi$ is 
bijective. Assume first that $x, y\in E$ such that $\varphi(x) =\varphi(y)$. 
We have $x=\widetilde{x}\oplus_{E} (x\ominus_{E} \widetilde{x})=%
\widetilde{y}\oplus_{E} (y\ominus_{E} \widetilde{y})=y$. Hence 
$\varphi$ is injective. Let $(x_S, x_M)\in \Sh(E) \times \Mea(E)$ and 
$x_M\in h(x_S')$. This yields that $x=x_S\oplus_{E} x_M$ exists and 
evidently by Lemma \ref{jpy2}, (i) $\widetilde{x}=x_S$ and $x\ominus_{E} \widetilde{x}=x_M$.
It follows that $\varphi$ is surjective. Moreover, 
$\varphi(0_{E})=(0_E,0_E)=0_{\Tea(E)}$ and\/ $\varphi(1_{E})=(1_E,0_E)=1_{\Tea(E)}$.

Now, let us check that, for all $x, y\in E$, $x\oplus_E y$ 
is defined iff $\varphi(x)\oplus_{\Tea(E)} \varphi(y)$ is defined 
in which case $\varphi(x\oplus_E y)=\varphi(x)\oplus_{\Tea(E)} \varphi(y)$.  
For any $x, y, z\in E$ we obtain 
\begin{center}
\begin{tabular}{@{}c r c l}
\multicolumn{4}{@{}l}{$z=x\oplus_E y$  $\iff$}\\ 
\multicolumn{4}{@{}l}{$z=\bigl(\widetilde{x}\oplus_E (x\ominus_{E} \widetilde{x})\bigr)%
\oplus_E \bigl(\widetilde{y}\oplus_E (y\ominus_{E} \widetilde{y})\bigr)$ 
 $\iff$} \\
\multicolumn{4}{@{}l}{$z=(\widetilde{x}\oplus_E \widetilde{y})%
\oplus_E \bigl((x\ominus_{E} \widetilde{x})\oplus_E (y\ominus_{E} \widetilde{y})\bigr)$ \ $\iff$} \ %
\\
\multicolumn{4}{@{}l}{by Lemma \ref{pommeag}\ $(\exists u\in E)\ %
u=S(x\ominus_{E} \widetilde{x}, y\ominus_{E} \widetilde{y})$  and}\\ %
\multicolumn{4}{@{}l}{$\begin{array}{@{}r@{}c@{}l}
z=(\widetilde{x}\oplus_E \widetilde{y})%
\oplus_E \biggl(u&\oplus_{E}&\bigl((x\ominus_{E} \widetilde{x})\ominus_{E} (u\wedge (x\ominus_{E} \widetilde{x}))\bigr)\\
&\oplus_{E}&\bigl((y\ominus_{E} \widetilde{y})\ominus_{E} (u\wedge (y\ominus_{E} \widetilde{y}))\bigr)\biggr)\ 
\end{array}$ }\\
$\iff$&\multicolumn{3}{@{}l}{$(\exists u\in E)\ %
u=S(x\ominus_{E} \widetilde{x}, y\ominus_{E} \widetilde{y})$  and}\\ %
\multicolumn{4}{@{}l}{$\begin{array}{r@{}c@{}l}
z=(\widetilde{x}\oplus_E \widetilde{y}%
\oplus_E u)&\oplus_{E}&\biggl(\bigl((x\ominus_{E} \widetilde{x})\ominus_{E} (u\wedge (x\ominus_{E} \widetilde{x}))\bigr)\\
&\oplus_{E}&\phantom{(}\bigl((y\ominus_{E} \widetilde{y})\ominus_{E} (u\wedge (y\ominus_{E} \widetilde{y}))\bigr)\biggr)\ 
\end{array}$ }\\
$\iff$&\multicolumn{3}{@{}l}{$(\exists u\in E)\ %
u=S(x\ominus_{E} \widetilde{x}, y\ominus_{E} \widetilde{y})$  and}\\ %
\multicolumn{4}{@{}l}{$\begin{array}{r@{}c@{}l}
z=(\widetilde{x}\oplus_{\Sh(E)} \widetilde{y}%
&\oplus_{\Sh(E)}& u) \\
&\oplus_{E}&\biggl(\bigl((x\ominus_{E} \widetilde{x})\ominus_{\Mea(E)} %
(u\wedge (x\ominus_{E} \widetilde{x}))\bigr)\\
&\oplus_{\Mea(E)}&\phantom{(}\bigl((y\ominus_{E} \widetilde{y})\ominus_{\Mea(E)} %
(u\wedge (y\ominus_{E} \widetilde{y}))\bigr)\biggr)\ 
\end{array}$ }\\
$\iff$&\multicolumn{3}{@{}l}{$(\widetilde{x}, x\ominus_{E} \widetilde{x}) %
\oplus_{\Tea(E)} (\widetilde{y}, y\ominus_{E} \widetilde{y})$ is defined and}\\
\multicolumn{4}{@{}l}{%
\begin{tabular}{@{}r@{}c@{}l}$\varphi(z)$&=&%
$\biggl(\widetilde{x}\oplus_{\Sh(E)}\widetilde{y}\oplus_{\Sh(E)}%
S(x\ominus_{E} \widetilde{x}, y\ominus_{E} \widetilde{y}), %
\bigl((x\ominus_{E} \widetilde{x}) \ominus $\\%
&&($S(x\ominus_{E} \widetilde{x}, y\ominus_{E} \widetilde{y})%
{\wedge} (x\ominus_{E} \widetilde{x}))\bigr)\oplus_{\Mea(E)} \bigl((y\ominus_{E} \widetilde{y}) $\\%
&&$\ominus%
(S(x\ominus_{E} \widetilde{x}, y\ominus_{E} \widetilde{y})%
\wedge (y\ominus_{E} \widetilde{y}))\bigr)\biggr)$\\
&=&$(\widetilde{x}, x\ominus_{E} \widetilde{x}) %
\oplus_{\Tea(E)} (\widetilde{y}, y\ominus_{E} \widetilde{y})=%
\varphi(x)\oplus_{\Tea(E)} \varphi(y)$.
\end{tabular}}
\end{tabular}
\end{center}

Altogether, $\Tea(E)=(\Tea(E), \oplus_{\Tea(E)}, 0_{\Tea(E)},$ $1_{\Tea(E)})$ 
is an effect algebra and the mapping $\varphi:E\to \Tea(E)$     
is an isomorphism of effect algebras.
\end{proof}

The Triple Representation Theorem then follows immediately.

\begin{remark}\label{larem}{\rm Recall that our method may be also used in the case of 
complete lattice effect algebras as a substitute of the method from 
\cite{jenca}. Moreover, since any  homogeneous orthocomplete effect algebra $E$ is both  meager-orthocomplete and 
sharply dominating  the Triple Representation Theorem is valid within the class of homogeneous orthocomplete effect algebras 
which was an open question asked by Jen\v{c}a in \cite{jenca}.}
\end{remark}

\section*{Acknowledgements}  J. Paseka gratefully acknowledges Financial Support 
of the  Ministry of Education of the Czech Republic
under the project MSM0021622409 and of Masaryk University under the grant 0964/2009. 
Both authors acknowledge the support by ESF Project CZ.1.07/2.3.00/20.0051
Algebraic methods in Quantum Logic of the Masaryk University.

\end{document}